\documentclass[11pt]{article}
\usepackage{epigamath}

\usepackage[notext]{kpfonts}
\usepackage{baskervald}

\setpapertype{A4}

\usepackage[english]{babel}

\usepackage[dvips]{graphicx}     

\title{\vspace{-0.7cm}Hamiltonian actions of unipotent groups on compact\\ K\"ahler manifolds}
\titlemark{Hamiltonian actions of unipotent groups}
\author{\vspace{0cm} Daniel Greb and Christian Miebach}
\authoraddresses{
\authordata{Daniel Greb}{\firstname{Daniel} \lastname{Greb}\\
\institution{Essener Seminar f\"ur Algebraische Geometrie und Arithmetik, 
Fakult\"at f\"ur Ma\-the\-matik, Universit\"at Duisburg--Essen, 45117 Essen, 
Germany}\\
\email{daniel.greb@uni-due.de}}\\
\authordata{Christian Miebach}{\firstname{Christian}
\lastname{Miebach}\\
\institution{Universit\'e du Littoral C\^ote d'Opale, EA 2797 - LMPA - Laboratoire de 
math\'ematiques pures et appliqu\'ees Joseph Liouville, F-62228 Calais, 
France}\\
\email{christian.miebach@univ-littoral.fr}}
}
\authormark{D. Greb and C. Miebach}
\date{\vspace{-5ex}} 
\journal{\'Epijournal de G\'eom\'etrie Alg\'ebrique} 
\acceptation{Received by the Editors on May 8, 2018, and in final form
on September 6, 2018. \\ Accepted on October 6, 2018.}

\newcommand{\articlereference}{Volume 2 (2018), Article Nr. 10}

 \usepackage[all]{xy}

\allowdisplaybreaks

\numberwithin{equation}{numsection}

\newtheorem{thm}{Theorem}[section]
\newtheorem{prop}[thm]{Proposition}
\newtheorem{lem}[thm]{Lemma}
\newtheorem{cor}[thm]{Corollary}
\newenvironment{thm*}{\paragraph{Theorem}}{\bigskip}

\newtheorem{defn}[thm]{Definition}

\newtheorem{ex}[thm]{Example}
\newtheorem{rem}[thm]{Remark}

\newcommand{\mbb}[1]{\mathbb{#1}}

\newcommand{\wh}[1]{\hat{#1}}
\newcommand{\ol}[1]{\overline{#1}}
\newcommand{\lie}[1]{{\mathfrak{#1}}}
\newcommand{\abs}[1]{\lvert #1\rvert}
\newcommand{\norm}[1]{\lVert #1\rVert}
\newcommand{\hq}{/\hspace{-0.09cm}/}

\DeclareMathOperator{\id}{id}

\DeclareMathOperator{\Ad}{Ad}

\DeclareMathOperator{\Aut}{Aut}
\DeclareMathOperator{\Alb}{Alb}
\DeclareMathOperator{\Lie}{Lie}

\DeclareMathOperator{\reg}{reg}

\DeclareMathOperator{\cone}{cone}

\DeclareMathOperator{\Spec}{Spec}

\DeclareMathOperator{\inn}{int}

\newcommand{\acts}{\mbox{\,\raisebox{0.26ex}{\tiny{$\bullet$}}\,}}

\begin{document}


\maketitle



\begin{prelims}

\vspace{-0.55cm}

\def\abstractname{Abstract}
\abstract{We study meromorphic actions of unipotent complex Lie groups on compact K\"ahler 
manifolds using moment map techniques. We introduce natural stability conditions 
and show that sets of semistable points are Zariski-open and admit geometric 
quotients that carry compactifiable K\"ahler structures obtained by symplectic 
reduction. The relation of our complex-analytic theory to the work of 
Doran--Kirwan regarding the Geometric Invariant Theory of unipotent group 
actions on projective varieties is discussed in detail.}



\vspace{0.15cm}

\languagesection{Fran\c{c}ais}{%

\textbf{Titre. Actions hamiltoniennes des groupes unipotents sur les vari\'et\'es k\"ahl\'eriennes compactes} \commentskip \textbf{R\'esum\'e.} Nous \'etudions les actions m\'eromorphes de groupes de Lie complexes unipotents sur les vari\'et\'es k\"ahl\'eriennes compactes en utilisant des techniques de type application moment. Nous introduisons des conditions de stabilit\'e naturelles et nous montrons que l'ensemble des points semi-stables forme un ouvert de Zariski et admet des quotients g\'eom\'etriques munis de structures k\"ahl\'eriennes compactifiables obtenues par r\'eduction symplectique. Le parall\`ele entre notre th\'eorie analytique complexe et les travaux de Doran--Kirwan concernant la Th\'eorie G\'eom\'etrique des Invariants des actions de groupes unipotents sur les vari\'et\'es projectives est discut\'e en d\'etails.}

\end{prelims}


\newpage

\setcounter{tocdepth}{1} \tableofcontents

\section*{Introduction}

\setcounter{numsection}{0}

Since the fundamental work of Mumford~\cite{MumfordGIT}, Kirwan~\cite{Kir2}, 
Guillemin--Sternberg \cite{GS_Mult}, and others, moment map geometry has become 
one of the most important tools for studying actions of \emph{complex-reductive 
Lie groups} $G=K^{\mathbb{C}}$ on K\"ahler manifolds. Given a Hamiltonian 
$G$-manifold, i.e., a K\"ahler $G$-manifold $(X, \omega_X)$ admitting a moment 
map $\mu\colon X \to \Lie(K)^*$ for the $K$-action, by the work of 
Heinzner--Loose~\cite{ReductionOfHamiltonianSpaces}, 
Heinzner--Huckleberry--Loose \cite{Extensionofsymplectic}, and 
Sjamaar~\cite{Sj2}, the set of $\mu$-semistable points $X^{ss}_G(\mu) := \{x \in 
X \mid \overline{G\acts x} \cap \mu^{-1}(0) \neq \emptyset \}$ admits an 
analytic Hilbert quotient, i.e., a $G$-invariant holomorphic Stein map 
$\pi\colon X^{ss}_G(\mu)\to X^{ss}_G(\mu)\hq G$ onto a K\"ahlerian complex 
space $X^{ss}_G(\mu)\hq G$ with structure sheaf 
$\mathscr{O}_{X^{ss}_G(\mu)/\negthickspace / 
G}=(\pi_*\mathscr{O}_{X^{ss}_G(\mu)})^G$; see also the survey \cite{HH2}. If $X$ 
is projective algebraic, and if the K\"ahler form $\omega_X$ as well as the 
moment map $\mu$ are induced by an embedding of $X$ into some projective space, 
both the set $X^{ss}_G(\mu)$ of semistable points and the quotient 
$X^{ss}_G(\mu)\hq G$ are the ones constructed via Geometric Invariant Theory 
(GIT). This theory crucially uses the fact that complex-analytic objects on $X$ 
can be averaged over the compact group $K$ to produce $G$-invariant objects, 
which then can be used to construct the quotient. 

On the other hand, actions of \emph{unipotent groups} on (compact) K\"ahler 
manifolds appear naturally in a number of contexts and play an important role in 
K\"ahler geometry. By a fundamental result of Lichnerowicz and Matsushima, a 
given compact K\"ahler manifold $X$ can admit a constant scalar curvature 
K\"ahler metric only if the Lie sub-algebra of all holomorphic vector fields 
having a zero is reductive, see e.g.~\cite[Proposition~4.18 and 
Remark~4.12]{Sze}. In other words, unipotent subgroups of $\mathrm{Aut}(X)$ 
appear as obstructions to the existence of such metrics. In a related direction, 
the paper \cite{CD} proposes a way to produce canonical destabilising 
test-configurations (showing $K$-unstability of $X$) from non-reductive 
subgroups of the automorphism group. 

Motivated by these and other moduli-theoretic questions, Doran and Kirwan in 
\cite{DorKir} started to study actions of unipotent algebraic groups $N$ on 
projective manifolds $X$ linearised in very ample line bundles, using 
invariant-theoretic methods on the one hand and Geometric Invariant Theory for 
related actions of reductive groups $G$ containing $N$ on twisted products $G 
\times_N X$ on the other hand. When thinking about the relation of their work to 
K\"ahler geometry and moment maps, one encounters three basic questions: 

\begin{enumerate}
 \item[(a)] What is the correct analogue of a ``linear'' action in K\"ahler 
geometry? 
 \item[(b)] Given a compact K\"ahler $N$-manifold $X$, how can one produce 
K\"ahler metrics on the non-compact twisted product $G\times_N X$ ?
 \item[(c)] If (b) has a positive answer, can one use moment map geometry on 
the non-compact Hamiltonian $G$-manifold to produce quotients for the $N$-action 
on $X$ with good geometric and complex-analytic properties?
\end{enumerate}

With a different set of problems in mind, Question (a) has been solved a rather 
long time ago by Fujiki~\cite{Fuj}, Lieberman~\cite{Lie}, and 
Sommese~\cite{Som}: meromorphic actions and more generally actions for which the 
induced action on the Albanese torus is trivial were already called ``linear'' 
or ``projective'' in \emph{loc.~cit.}, and it turns out that also with a view 
towards moment map geometry and K\"ahlerian quotient theory, these are the 
correct conditions to impose, see~Remark~\ref{rem:reductive_equivalence}. For 
actions of reductive Lie groups this was observed by Huckleberry--Wurzbacher 
\cite[Remark on page 262]{HW} and Fujiki~\cite[Lemma 2.1]{Fuj2}. Note that the 
Lie algebra of the group of all automorphisms acting trivially on Albanese 
consists exactly of those holomorphic vector fields having a zero, e.g.~see 
\cite[Proposition 6.8]{Fuj}, which connects the question of ``linearity'' to the 
one concerning the existence of extremal K\"ahler metrics discussed above.

\subsection*{Main results}

As our first contribution, using a criterion of Blanchard and and properties of 
unipotent groups, with respect to Question (b) we prove the following result.

\paragraph{Theorem (Theorem~\ref{Thm:ExistenceExtensions}).}
{\em Let $N$ be a unipotent subgroup of the simply-connected complex semisimple Lie 
group $G$. For a connected compact K\"ahler manifold $(X, \omega_X)$ endowed with a 
holomorphic $N$-action the following statements are equivalent.
\begin{enumerate}[\rm (1)]
\item There exists a Hamiltonian $G$-extension $(Z,\omega_Z)$ of the $N$-action 
on $X$.
\item The $N$-action on $X$ is meromorphic.
\item The twisted product $G\times_NX$ is K\"ahler.
\end{enumerate}
Moreover, given a meromorphic $N$-action on $X$ we can always find a 
Hamiltonian $G$-extension that is a $G$-equivariant compactification of 
$G\times_NX$.}

\bigskip

Here, a \emph{Hamiltonian $G$-extension} of the $N$-action on
$(X,\omega_X)$ consists of a connected compact Hamiltonian $G$-manifold 
$(Z,\omega_Z)$ and an $N$-equivariant embedding $\iota\colon X\hookrightarrow Z$ 
such that for the de Rham cohomology classes associated with the K\"ahler forms 
we have $\iota^*[\omega_Z]=[\omega_X]$. This theorem links the condition 
``meromorphic'' to moment map geometry of the $G$-action on $G \times_N X$, and 
therefore opens the door to using the complex-reductive theory for the 
construction of quotients of $X$ with respect to the $N$-action.

Once the existence of Hamiltonian $G$-extensions is established, these can be 
used to define the set of \emph{$N$-semistable points} $X^{ss}_N[\omega_X]$ with 
respect to the given K\"ahler class $[\omega_X] \in H^2(X, \mathbb{R})$, after
one has chosen a suitable K\"ahler form on the quasi-affine 
homogeneous space $G/N$. While we show in Theorem~\ref{Thm:Independence}, using 
Hodge-theoretic arguments as well as the relation of $G$-semistability to 
$K$-invariant strictly plurisubharmonic exhaustion functions, that this 
definition does not depend on the chosen $G$-extension, the choice of metric on 
$G/N$ influences semistability in a subtle way, as we explore in great detail in 
Section~\ref{subsect:discussion_of_choice}. The problems that occur are closely 
related to the ones encountered in the algebraic situation when searching for 
various kinds of ``reductive envelopes'', cf.~\cite[Sections~5.2 and 
5.3]{DorKir}, and can be traced back to the fact that in general there is no 
choice of metric on $G/N$ so that the corresponding moment map is 
proper or admissible in the sense of \cite{Sj2}. A detailed comparison of the 
moment map approach introduced here and the GIT approach of Doran--Kirwan is 
presented in Section~\ref{subsect:algebraic_actions}.

Finally, regarding Question (c), we establish that the set of $N$-semistable 
points indeed has a number of very desirable complex-geometric properties. The 
following result summarises the content of 
Theorem~\ref{thm:existence_of_geometric_quotient}, 
Proposition~\ref{prop:compactI}, Theorem~\ref{thm:Zopen}, and 
Theorem~\ref{thm:reducedKaehler}. 

\paragraph{Theorem.}
{\em  Let $(X, \omega_X)$ be a compact K\"ahler manifold endowed with a meromorphic 
  $N$-action. Then, the following holds.
  \begin{enumerate}
   \item[\rm (a)] The set $X^{ss}_N[\omega_X]$ of semi\-stable points admits a 
geometric quotient $\pi\colon X^{ss}_N[\omega_X] \to X^{ss}_N[\omega_X] / N$ by 
the $N$-action. In fact, $\pi$ is a principal $N$-fibre bundle and 
$X^{ss}_N[\omega_X] / N =:Q$ is smooth. 
   \item[\rm (b)] The definition of semistability naturally induces an open 
embedding $\phi\colon Q=X^{ss}_N[\omega_X] / N \hookrightarrow \overline{Q}$ of $Q$ into a compact 
complex space $\overline{Q}$ such that $\overline{Q} \setminus \phi( Q )$ is analytic.
 \item[\rm (c)] The set $X^{ss}_N[\omega_X]$ of semi\-stable points is Zariski-open 
in $X$. Moreover, the quotient map $\pi\colon X^{ss}_N[\omega_X] \to 
Q$ extends to a meromorphic map $\pi\colon X \dasharrow 
\overline{Q}$.
 \item[\rm (d)] There exists a K\"ahler structure $\omega_{\ol{Q}}$ on $\ol{Q}$ 
whose restriction $\omega_Q = \omega_{\ol{Q}}|_Q^{}$ to $Q \hookrightarrow 
\ol{Q}$ is smooth and fulfils
 \[ [\pi^*{\omega_Q}] = [\omega_X|_{X^{ss}_N[\omega_X]}] 
 \in H^2(X^{ss}_N[\omega_X],\, \mathbb{R}).\]
  \end{enumerate}}

\subsection*{Future directions}

With the fundamental results of a K\"ahlerian quotient theory for meromorphic 
actions of unipotent groups established, interesting questions include whether 
despite the difficulties presented by the examples collected in Section 4.3 
under certain additional conditions the set of semistable points can be shown to 
be independent of further choices, whether in certain applications there are 
natural K\"ahler metrics on the homogeneous space $G/N$ leading to a quotient 
that is as well-adapted as possible to the given geometric situation at hand, 
and whether given a compact K\"ahler manifold $(X, \omega_X)$ with non-trivial 
non-reductive part in the automorphism group (obstructing the existence of 
special metrics) one can use the K\"ahlerian quotient theory for this unipotent 
group in order to produce a K\"ahlerian complex space/manifold where said 
obstruction vanishes and special metrics might exist. 

\subsection*{Organisation of the paper}

In the first two sections we review basic facts about meromorphic actions and 
GIT on K\"ahler manifolds via moment maps. In Section~\ref{Section:Extensions} 
we prove our first main result, Theorem~\ref{Thm:ExistenceExtensions}. The 
difficulties one encounters in defining the set of semistable points for 
holomorphic actions of unipotent groups, the actual definition of 
semistability, as well as fundamental properties of the set of semistable points 
are discussed in Section~\ref{Section:Semistable}, while in the final 
Section~\ref{sect:properties} we prove the second main result concerning the 
properties of semistable quotients for unipotent group actions.

\subsection*{Acknowledgements}

The first author wants to thank the Laboratoire de Math\'ematiques Pures et 
Appliqu\'ees at Universit\'e du Littoral C\^ote d'Opale as well as FRIAS for 
hospitality during research visits in the spring of 2016 and 2018, 
respectively. During the preparation of this paper, he has been partially 
supported by the DFG-Collaborative Research Center SFB/TR 45 ``Periods, moduli 
spaces and arithmetic of algebraic varieties''. The second author gratefully 
acknowledges the hospitality of the mathematical departments of the universities 
of Duisburg-Essen and Freiburg as well as of the FRIAS during the spring of 
2016, 2017 and 2018. Both authors would like to thank Peter Heinzner for 
fruitful and interesting discussions.

\subsection*{Global conventions}

We work over the field $\mathbb{C}$ of complex numbers. A \emph{complex space} 
is a reduced complex space with countable topology. \emph{Analytic
subsets} are assumed to be closed. \emph{Manifolds} are assumed to be connected.

\section{Meromorphic group actions}\label{Section:mero}

Let us review some facts about meromorphic group actions from~\cite{Fuj} that 
we will use freely in the following. Lieberman obtained essentially the same 
results in \cite{Lie}. In this section, $G$ denotes a complex Lie group. 

A \emph{meromorphic structure} on the complex Lie group $G$ is a compactification 
$\ol{G}$ together with a meromorphic mapping 
$\ol{\mu}\colon\ol{G}\times\ol{G}\to\ol{G}$ which extends the group 
multiplication of $G$ such that $\ol{\mu}$ is holomorphic on 
$(G\times\ol{G})\cup(\ol{G}\times G)$. In other words, the compactification 
$\ol{G}$ is $(G\times G)$-equivariant. Moreover, the map $G\to G$, $g\mapsto 
g^{-1}$, extends to a meromorphic map $\ol{G}\to\ol{G}$.

Let us fix a meromorphic structure on $G$. A complex subgroup $H$ of $G$ is 
\emph{meromorphic} if the topological closure $\ol{H}$ of $H$ in $\ol{G}$ is 
analytic.

\begin{rem}{\rm
If $G$ is linear algebraic, we may and will choose $\ol{G}$ to be a 
projective manifold. The meromorphic subgroups of $G$ are then precisely the 
algebraic subgroups of $G$.}
\end{rem}

Let $X$ be a complex space endowed with a holomorphic $G$-action. This 
$G$-action is \emph{meromorphic} if the action map $G\times X\to X$ extends to 
a meromorphic map $\ol{G}\times X\to X$. For meromorphic actions on compact 
K\"ahler manifolds (in fact, on reduced compact complex spaces of class 
$\mathcal{C}$) we have the following quotient theorem, 
see~\cite[Theorem~4.1]{Fuj}.

\begin{thm}[Quotient Theorem]
Let $X$ be a compact K\"ahler manifold on which $G$ acts meromorphically. Then, 
there exist a compact complex space $Y$ and a $G$-invariant surjective 
meromorphic map $\pi\colon X\dasharrow Y$ such that the following universal 
property is satisfied. If $\pi'\colon X\dasharrow Y'$ is another $G$-invariant 
meromorphic map to a compact complex space $Y'$, then there exists a unique 
meromorphic map $m\colon Y \dasharrow Y'$ such that $m\circ\pi=\pi'$. In this 
situation we call $\pi\colon X\dasharrow Y$ a \emph{meromorphic quotient} for 
the $G$-action on $X$.
\end{thm}

\begin{ex}{\rm
Let $H$ be an algebraic subgroup of $G$. Then there exists a $G$-equivariant 
smooth projective compactification $\ol{G/H}$ of $G/H$. Moreover, the 
$H$-principal bundle $\pi\colon G\to G/H$ extends to a rational map $\ol{\pi} 
\colon\ol{G}\dasharrow\ol{G/H}$. Then $\ol{\pi}$ is a meromorphic quotient for the 
$H$-action on $\ol{G}$.}
\end{ex}

In fact, by inspecting Fujiki's construction, one obtains the following more 
precise result that has appeared in several places in the literature; see for 
example \cite[Theorem 0.2.2]{BBS}, as well as \cite[Proposition 3.1]{Greb}, 
\cite[Section 3]{Hu}, and the references given there for the analogous result in 
the algebraic category.

\begin{prop}\label{Chow}
Let $X$ be a compact K\"ahler manifold on which $G$ acts meromorphically. Then, 
there exists an irreducible, compact analytic subset $Q_F$ of the cycle space of 
$X$, a $G$-invariant meromorphic map $\pi_F\colon X \dasharrow Q_F$, and a 
$G$-invariant Zariski-open subset $U_F \subset \mathrm{dom}(\pi_F)$, called a 
\emph{Fujiki set} of $X$, such that
\begin{enumerate}[\rm (1)]
\item $\pi_F\colon X \dasharrow Q_F$ is a meromorphic quotient for the 
$G$-action on $X$, 
\item $U_F \subset X_{\mathrm{gen}} := \{x \in X \mid \dim G\acts x = m \text{ 
is maximal}\}$,
\item for all $u \in U_F$, we have $\pi_F (u) = \overline{G\acts u}$, 
considered as a (reduced) cycle of $X$,
\item $\pi_F(U_F)$ is smooth and Zariski-open in $Y$, and the restriction 
$\pi_F|_{U_F}\colon U_F \to \pi_F(U_F)$ is a geometric quotient.
\end{enumerate}
\end{prop}

We call $\pi_F\colon X \dasharrow Q_F$ a \emph{Fujiki quotient} of $X$ by $G$.

\section{Hamiltonian $G$-spaces}

Let $G=K^\mbb{C}$ be a complex reductive group with maximal compact subgroup 
$K$. In this section we review the definition and some properties of Hamiltonian 
$G$-manifolds. A general reference for the complex-analytic theory of moment 
maps on K\"ahler manifolds is~\cite{HH2}.

\subsection{Moment maps}

Let $(Z,\omega_Z)$ be a K\"ahler manifold with K\"ahler form $\omega_Z$. Suppose 
that $G$ acts holomorphically on $Z$ such that $\omega_Z$ is $K$-invariant. A 
\emph{moment map} for the $K$-action on $Z$ is a $K$-equivariant smooth map 
$\mu\colon Z\to\lie{k}^*$, where $K$ acts via the coadjoint representation on 
the dual $\lie{k}^*$ of its Lie algebra, such that
\begin{equation}\label{Eqn:momentcondition}
d\mu^\xi=\iota_{\xi_Z}\omega_Z.
\end{equation}
Here, for every $\xi\in\lie{k}$, we write $\mu^\xi\in\mathcal{C}^\infty(Z)$ for 
the function defined by $\mu^\xi(z)=\mu(z)(\xi)$, and $\xi_Z$ for the vector 
field on $Z$ whose flow is given by $(t,z)\mapsto\exp(-t\xi)\acts z$, and 
$\iota_{\xi_Z}\omega_Z$ for the contraction of $\omega_Z$ by $\xi_Z$.

If a moment map for the $K$-action on $Z$ exists, we say that $(Z,\omega_Z)$ is 
a \emph{Hamiltonian $G$-manifold}. The notions introduced above make sense for 
actions of $G$ on K\"ahlerian complex spaces, see for example \cite[Sections 3.1 
and 3.2]{PaHq} and the references given there; in this setup one speaks about 
\emph{Hamiltonian $G$-spaces}.

\begin{rem}\label{rem:semisimple_always_Hamiltonian}{\rm
If $G$ is semisimple, then every K\"ahler manifold $(Z,\omega_Z)$ endowed with a 
holomorphic $G$-action such that $\omega_Z$ is $K$-invariant is Hamiltonian. 
Moreover, in this case the moment map $\mu\colon Z\to\lie{k}^*$ is unique, 
see~\cite[Proposition~24.1]{GS} and the discussion following this proposition.}
\end{rem}

\begin{rem}\label{rem:reductive_equivalence}{\rm
For a \emph{compact} K\"ahler manifold $(Z,\omega_Z)$ endowed with a 
holomorphic action of the connected complex reductive Lie group $G=K^\mathbb{C}$ 
and $K$-invariant K\"ahler form $\omega_Z$ the following statements are 
equivalent:
\begin{enumerate}[\rm (1)]
\item $Z$ is a Hamiltonian $G$-manifold.
\item $G$ acts meromorphically on $Z$ in the sense of~\cite{Fuj}, see also 
Section~\ref{Section:mero}.
\item $G$ acts trivially on the Albanese torus $\Alb(Z)$ of $Z$.
\end{enumerate}
The equivalence $(1)\Longleftrightarrow(3)$ was observed in~\cite[Lemma~2.1 and 
subsequent remark]{Fuj2}. The implication $(2)\Longrightarrow(3)$ follows from 
\cite[Lemma~3.8]{Fuj} since every complex reductive group is linear algebraic. 
The last implication $(3)\Longrightarrow(2)$ follows 
from~\cite[Proposition~I]{Som}, see also~\cite[Proposition~6.10]{Fuj}.}
\end{rem}

For later use we record an elementary result on the moment image of a  
Hamiltonian $G$-manifold $(Z,\omega_Z)$ with moment map $\mu\colon 
Z\to\lie{k}^*$: If the interior of $\mu(Z)$ relative to $\lie{k}^*$ is 
non-empty, then, due to Sard's Theorem, there exists a point $z\in Z$ such that 
$d\mu_z$ is surjective. Since condition~\eqref{Eqn:momentcondition} implies
\begin{equation}\label{eq:magic_formula}
\ker d\mu_z=(\lie{k}\acts z)^{\perp_{\omega_Z}}=
\{\xi_Z(z)\mid \xi\in\lie{k}\}^{\perp_{\omega_Z}},
\end{equation}
cf.~\cite[Section~2.3]{HH2}, and therefore that the rank of $\mu$ in $z$ 
coincides with $\dim K\acts z$, we conclude that $K_z$ is finite. Conversely, if 
$K_z$ is finite, then $\mu$ is a submersion in $z$, hence the image of $\mu$ has 
interior points in $\lie{k}^*$. We summarize this discussion in the following

\begin{lem}\label{Lem:IntPoints}
Suppose that $(Z,\omega_Z)$ is a $G$-connected\footnote{We say that $Z$ is 
$G$-connected if it cannot be written as the disjoint union of two non-empty $G$-stable 
closed subsets.} Hamiltonian $G$-manifold. Then $\mu(Z)$ has non-empty interior 
in $\lie{k}^*$ if and only if $K$ acts with generically finite isotropy on $Z$.
\end{lem}

\subsection{The set of semistable points}

The set of semistable points is defined by
\begin{equation*}
Z^{ss}_G(\mu):=\{z\in Z\mid \ol{G\acts z}\cap\mu^{-1}(0)\not=
\emptyset\}.
\end{equation*}
The $G$-invariant set $Z^{ss}_G(\mu)$ is open and can be characterized in the 
following way. For $z\in Z^{ss}_G(\mu)$ consider the inclusion $\iota\colon 
\ol{G\acts z}\cap Z^{ss}_G(\mu) \hookrightarrow Z$. Then 
$\iota^*\omega_Z=i\partial\ol{\partial}\rho$ for some strictly plurisubharmonic 
exhaustion function $\rho$ and $\iota^*\mu=\mu_\rho$ where 
$\mu_\rho\colon\ol{G\acts z}\cap Z^{ss}_G(\mu)\to\lie{k}^*$ is given by 
\begin{equation}\label{Eqn:murho}
\mu_\rho(z)(\xi):=d\rho_z(J\xi_Z(z)),
\end{equation}
where $J$ denotes the complex structure of $Z$, see~\cite[Section~3]{HH1}. 

If $Z$ is compact and if $G$ is semisimple, then $Z^{ss}_G(\mu)$ depends only on 
the K\"ahler class $[\omega_Z]$, see~\cite[p.~71]{HH1}. Since we generalize 
this result later on for Hamiltonian actions of unipotent groups, we repeat its 
proof here for the readers' convenience.

\begin{prop}\label{Prop:Semistable}
Let $(Z,\omega_Z)$ be a compact Hamiltonian $G$-manifold with moment map $\mu$ 
and let $\omega_Z'$ be another $K$-invariant K\"ahler form on $Z$ such that 
$[\omega_Z]=[\omega_Z']\in H^2(Z,\,\mbb{R})$. Then there exists a 
moment map $\mu'$ for the $K$-action on $(Z,\omega_Z')$ such that 
$Z^{ss}_G(\mu) =Z^{ss}_G(\mu')$.
\end{prop}

\begin{proof}
Since $Z$ is compact K\"ahler, the $\partial\ol{\partial}$-lemma implies 
$\omega_Z'=\omega_Z+i\partial\ol{\partial}\varphi$ for some 
$\varphi\in\mathcal{C}^\infty(Z)$. Defining $\mu_\varphi\colon Z\to\lie{k}^*$ in 
the same way as in Equation~\eqref{Eqn:murho} one checks directly that the map 
$\mu':=\mu+\mu_\varphi$ is a moment map for the $K$-action on $(Z,\omega_Z')$.

Now suppose that $z\in Z^{ss}_G(\omega_Z)$. As noted above, we have 
$\iota^*\omega_Z=i\partial\ol{\partial}\rho$ and $\iota^*\mu=\mu_\rho$ for some 
strictly plurisubharmonic function $\rho$ on $\ol{G\acts z}\cap Z^{ss}_G(\mu)$ 
where $\iota\colon \ol{G\acts z}\cap Z^{ss}_G(\mu)\hookrightarrow Z$ is the 
inclusion. Since $Z$ is assumed to be compact, the function $\varphi$ is 
bounded, hence $\rho+\iota^*\varphi$ is still an exhaustion function on 
$\ol{G\acts z}\cap Z^{ss}_G(\mu)$. Consequently, $\iota^*\mu'=
\mu_{\rho+\iota^*\varphi}$ has a zero on $\ol{G\acts z}\cap Z^{ss}_G(\mu)$, 
which implies $z\in Z^{ss}_G(\mu')$. The converse inclusion follows by symmetry.
\qed
\end{proof}

\begin{rem}{\rm
If $G$ is semisimple and if $(Z,\omega_Z)$ is a compact Hamiltonian 
$G$-manifold, the moment map for the $K$-action on $Z$ is unique, see 
Remark~\ref{rem:semisimple_always_Hamiltonian} above. Due to 
Proposition~\ref{Prop:Semistable}, the set of semistable points therefore 
depends only on $[\omega_Z]\in H^2(Z,\,\mbb{R})$. In this case we thus 
write $Z^{ss}_G[\omega_Z]$ instead of $Z^{ss}_G(\mu)$.}
\end{rem}

\subsection{Analytic Hilbert quotients}\label{subsect:aHq_properties}

The importance of semistability stems from the fact that the set of semistable 
points admits the analogue of a good quotient in the analytic category:

Let $G$ be a complex reductive Lie group and $Z$ a complex space endowed with a 
holomorphic $G$-action. A complex space $Y$ together with a $G$-invariant 
surjective holomorphic map $\pi\colon Z \to Y$ is called an \emph{analytic 
Hilbert quotient}\footnote{In some places in the literature, the terminology 
\emph{semistable quotient} is used for the same concept.} of $Z$ by the action 
of $G$ if
\begin{enumerate}[\rm (1)]
 \item $\pi$ is a locally Stein map, and
 \item $(\pi_*\mathscr{O}_Z)^G = \mathscr{O}_Y$ holds.
\end{enumerate}
Here, \emph{locally Stein} means that there exists an open covering of $Y$ by 
open Stein subspaces $U_\alpha$ such that $\pi^{-1}(U_\alpha)$ is a Stein 
subspace of $Z$ for all $\alpha$; by $(\pi_*\mathscr{O}_Z)^G$ we denote the 
sheaf $U \mapsto \mathscr{O}_Z(\pi^{-1}(U))^G = \{f \in 
\mathscr{O}_Z(\pi^{-1}(U)) \mid f \;\; G\text{-invariant}\}$, $U$ open in
$Y$.

An analytic Hilbert quotient of a holomorphic $G$-space $Z$ is unique up to 
biholomorphism once it exists and we will denote it by $Z\hq G$. The following 
properties follow from the corresponding ones in the Stein case, where analytic 
Hilbert quotients always exist, see \cite{HeinznerGIT}: Two points $x,x' \in Z$ 
have the same image in $Z\hq G$ if and only if $\overline{G\acts x} \cap 
\overline{G\acts x'} \neq \emptyset$. For each $q \in Z\hq G$, the fibre 
$\pi^{-1}(q)$ contains a unique closed $G$-orbit $G\acts x$. The stabiliser 
$G_x$ of $x$ in $G$ is a complex reductive Lie group, see \cite{Mat}. If $A 
\subset X$ is a $G$-invariant analytic subset, then $\pi(A) \subset X\hq G$ is 
analytic, and $\pi|_A\colon A \to \pi(A)$ is an analytic Hilbert quotient.

The main results in the quotient theory for complex reductive group actions on 
K\"ahler spaces are summarised in the following theorem.

\begin{thm}[\cite{ReductionOfHamiltonianSpaces},
\cite{Extensionofsymplectic}, \cite{HH2}, \cite{Sj2}]\label{propertiesmomentumquotients}
Let $Z$ be a Hamiltonian $G$-space with K\"ahler form $\omega_Z$ and moment map 
$\mu\colon Z \to \mathfrak{k}^*$. Then,
\begin{enumerate}[\rm (1)]
\item $Z^{ss}_G(\mu)$ is open and $G$-invariant, and the analytic Hilbert 
quotient $\pi\colon Z^{ss}_G(\mu) \to Z^{ss}_G(\mu)\hq G$ exists,
\item the inclusion $\mu^{-1}(0) \hookrightarrow Z^{ss}_G(\mu)$ induces a 
homeomorphism $\mu^{-1}(0)/K \simeq Z^{ss}_G(\mu)\hq G$,
\item the complex space $Z^{ss}_G(\mu)\hq G$ carries a K\"ahler structure that 
is induced by symplectic reduction from $\omega_Z$ and that is smooth along a 
natural stratification of $Z^{ss}_G(\mu)\hq G$.
\end{enumerate}
\end{thm}

\subsection{Moment maps associated with representations and their 
images}\label{sect:moment_maps_for_reps}

Let $G$ be a connected semisimple complex Lie group acting linearly on a finite 
dimensional complex vector space $V$. If we equip $V$ with the $K$-invariant 
flat K\"ahler metric $\omega_V$ given by a $K$-invariant hermitian inner 
product $\langle \cdot,\cdot\rangle$, then a moment map $\mu_V\colon 
V\to\lie{k}^*$ for the $K$-action on $V$ is given by $\mu_V^\xi(v)= 
-\frac{i}{2}\langle \xi.v,v\rangle$ for every $\xi\in\lie{k}$. Note that 
$\omega_V=i\partial\ol{\partial}\rho$ where $\rho(v)=\norm{v}^2$ for all 
$x\in V$ and that $\mu_V=\mu_\rho$ in this case.

For any $v\in V$ consider the affine $G$-variety $\ol{G\acts v}$. The 
restriction of $\mu_V$ to $\ol{G\acts v}$ yields the moment map for the 
$K$-action on $\ol{G\acts v}$ associated with the strictly plurisubharmonic 
exhaustion function $\rho|_{\ol{G\acts v}}$. By abuse of notation we will 
denote the restricted K\"ahler form and moment map again by $\omega_V$ and 
$\mu_V$, respectively. For later use we record the following result of Sjamaar, 
see~\cite[Theorem~4.9, Lemma~4.10]{Sj}. For its statement we have to introduce a 
maximal torus $T$ of $K$ with Lie algebra $\lie{t}$, the choice of a positive 
Weyl chamber $\lie{t}^*_+$, and the corresponding set $\Lambda^+$ of dominant 
weights. For $\lambda\in\Lambda^+$ let $V_\lambda$ denote the irreducible 
$G$-representation with highest weight $\lambda$.

\begin{thm}\label{Thm:MomentImage}
The moment map $\mu_V\colon\ol{G\acts v}\to\lie{k}^*$ is proper and verifies
\begin{equation*}
\mu_V(\ol{G\acts v})\cap \lie{t}^*_+=\cone\{\lambda\in\Lambda^+\mid 
V_\lambda\text{ occurs in $\mbb{C}[\ol{G\acts v}]$}\}.
\end{equation*}
\end{thm}

In general it may be rather difficult to decide for which dominant weight 
$\lambda$ the irreducible representation $V_\lambda$ occurs in 
$\mbb{C}[G\acts v]$. Note that the inclusion $G\acts v\hookrightarrow 
\ol{G\acts v}$ gives an injective homomorphism of algebras $\mbb{C}[\ol{G\acts 
v}]\to \mbb{C}[G\acts v]\cong\mbb{C}[G]^{G_v}$. Therefore the situation is 
slightly easier for affine completions of quasi-affine homogeneous spaces $G/G_v$ 
for which the map $\mbb{C}[\ol{G\acts v}]\to\mbb{C}[G]^{G_v}$ is an 
isomorphism. 

To follow this train of thought, recall that an algebraic subgroup $H$ of $G$ 
is called \emph{Grosshans} if $G/H$ is quasi-affine and if the algebra 
$\mbb{C}[G]^H\cong \mbb{C}[G/H]$ is finitely generated. This is equivalent to 
the existence of a finite-dimensional $G$-representation space $W$ containing 
$G/H$ as an orbit $G\acts w$ such that the codimension of $\ol{G\acts 
w}\setminus G\acts w$ in $\ol{G\acts w}$ is at least $2$, 
see~\cite[Theorem~4.3]{Gr2}. Recall that an algebraic subgroup of $G$ is called 
\emph{unipotent} if it consists entirely of unipotent elements. If $N$ is a 
unipotent subgroup of $G$, then $G/N$ is always quasi-affine, see 
\cite[Corollary 1.5]{Gr2}, but not affine, see \cite{Mat}; however, not every 
such $N$ is Grosshans. If $N$ is the unipotent radical of a parabolic subgroup 
of $G$, then $N$ is Grosshans, see~\cite[Theorem~2.2]{Gr1}.

Suppose that the unipotent subgroup $N$ of $G$ is Grosshans. Then we have a 
canonical affine completion $\ol{G/N}^\text{a}=\Spec\mbb{C}[G]^N$. Since $N$ is 
contained in a maximal unipotent subgroup of $G$, we can deduce 
from~\cite[Example~4.19]{Sj} that
\begin{equation*}
\{\lambda\in\Lambda^+\mid V_\lambda\text{ occurs in $\mbb{C}[G]^N$}\} 
=\Lambda^+.
\end{equation*}
Consequently, by embedding $\ol{G/N}^\text{a}$ into any $G$-representation, we 
can find a K\"ahler form inducing a proper, surjective moment map for the 
$K$-action on $\ol{G/N}^\text{a}$.  Combining this observation with an 
application of Lemma~\ref{Lem:IntPoints} to the free $K$-action on $G/N\subset 
\ol{G/N}^\text{a}$ we obtain the following result.

\begin{lem}\label{Lem:GrosshansMomentImage}
Suppose that $N$ is a unipotent Grosshans subgroup of $G$. Then there exists a 
K\"ahler form $\omega_V$ on $G/N$ such that the image of the corresponding 
moment map $\mu_V\colon G/N\to\lie{k}^*$ is a $K$-invariant dense open subset 
of $\lie{k}^*_{\reg}$. Moreover, $\omega_V$ and $\mu_V$ extend to the canonical 
affine completion $\ol{G/N}^{\text{a}}$.
\end{lem}

As the following example shows, this result depends on the choice of the 
$G$-repre\-sen\-ta\-tion into which we embed $G/N$.

\begin{ex}\label{Ex:nonsurjectivemomentmap}{\rm
Let $G$ be a simply-connected semisimple complex Lie group. For any choice of 
$\lambda_1,\dotsc,\lambda_k\in\Lambda^+$ consider $v:=v_1+\dotsb+v_k \in 
V^*_{\lambda_1}\oplus\dotsb \oplus V^*_{\lambda_k}$ where $v_j\in 
V^*_{\lambda_j}$ is a highest weight vector. Then we have $\mbb{C}[\ol{G\acts 
v}] =\bigoplus_{\lambda\in M}V_\lambda$ where $M$ is the submonoid of 
$\Lambda^+$ generated by $\lambda_1,\dotsc,\lambda_k$, see~\cite[Theorem~6]{PV}. 
Consequently, $\mu_V(\ol{G\acts v})\cap \mathfrak{t}^*$ is the cone generated 
by $\lambda_1,\dotsc,\lambda_k$.

To construct an example where the image of $\ol{G\acts v}$ under the moment map 
is not the whole of $\mathfrak{k}^*$, let $G={\rm{SL}}(3,\mbb{C})$ and choose 
$\lambda_1=\varpi_1+\varpi_2$ and $\lambda_2=2\varpi_1+\varpi_2$ where 
$\varpi_1,\varpi_2$ are the fundamental weights of $G$. Note that 
$V_{\lambda_1}$ and hence also $V_{\lambda_1}^*$ are isomorphic to the adjoint 
representation of $G={\rm{SL}}(3,\mbb{C})$. It follows that $G_{v_1}$ is the 
connected subgroup of $G$ having as Lie algebra the semi-direct sum of the 
kernel of $\lambda_1$ in the chosen Cartan sub-algebra of $\lie{g}$ and the 
positive maximal unipotent sub-algebra $\lie{n}$ of $\lie{g}$. Since the Lie 
algebra of $G_{v_2}$ for $v_2\in V_{\lambda_2}$ contains $\lie{n}$ and since the 
kernels of $\lambda_1$ and $\lambda_2$ intersect only trivially, the Lie algebra 
of the stabiliser of $v=v_1+v_2$ must coincide with $\lie{n}$. In summary, we 
see that in the chosen setup $G_v$ is the unipotent radical of a Borel 
subgroup of $G$, hence Grosshans, and that $\mu_V\colon\ol{G\acts 
v}\to\lie{k}^*$ is not surjective.}
\end{ex}

\begin{rem}\label{rem:forms_on_unipotent_radicals}{\rm
If $N$ is the unipotent radical of a parabolic subgroup of $G$, a 
$G$-representation space $E$ containing $\ol{G/N}^\text{a}$ and a certain 
$K$-invariant hermitian inner product on $E$ are described in great detail 
in~\cite{Kir}, extending~\cite{GJS} which dealt with unipotent radicals of Borel 
subgroups. In this situation it is natural to equip $\ol{G/N}^\text{a}$ with the 
restriction of the associated flat K\"ahler form $\omega_E$ as above, since the 
associated symplectic structure coincides with the one obtained via symplectic 
implosion from the cotangent bundle $T^*K$.}
\end{rem}

\section{Hamiltonian $G$-extensions}\label{Section:Extensions}

Let $G=K^\mbb{C}$ be a complex reductive group with maximal compact subgroup 
$K$. Recall that a unipotent subgroup of $G$ is by definition an algebraic 
subgroup of $G$ consisting entirely of unipotent elements. Such groups are 
automatically nilpotent and connected, see~\cite[Chapter~3.2.2]{OV}, hence 
simply-connected. Let $N$ be such a unipotent subgroup of $G$. Since our focus 
lies on actions of $N$, \emph{we suppose from now on that $G$ is connected and 
semisimple.}  Due to the simply-connectedness of $N$, by lifting to the 
universal cover if necessary, we may and will often assume that $G$ is 
simply-connected as well.

\subsection{Meromorphic actions and Hamiltonian extensions}

We will explore the relation between meromorphic $N$-actions and Hamiltonian 
$G$-actions. 

\begin{defn}{\rm
Let $(X, \omega_X)$ be a connected compact K\"ahler manifold endowed with a 
holomorphic $N$-action. A \emph{Hamiltonian $G$-extension of (the $N$-action on) 
$(X,\omega_X)$} consists in a connected compact Hamiltonian $G$-manifold 
$(Z,\omega_Z)$ and an $N$-equivariant embedding $\iota\colon X\hookrightarrow Z$ 
such that for the de Rham cohomology classes associated with the K\"ahler forms 
we have
\begin{equation}\label{eq:pullback}\iota^*[\omega_Z]=[\omega_X].\end{equation}}
\end{defn}

\begin{rem}{\rm
As $G=K^\mathbb{C}$, and hence $K$, is assumed to be semisimple, it follows 
from the fact that integration over $K$ does not change the cohomology class of 
a given K\"ahler form and from Remark~\ref{rem:semisimple_always_Hamiltonian} 
that any $N$-equivariant embedding of $(X,\omega_X)$ into a compact K\"ahler 
$G$-manifold $(Z,\omega_Z)$ satisfying Equation~\eqref{eq:pullback} is 
automatically a Hamiltonian $G$-extension.}
\end{rem}

The definition is motivated by the following example and the role it plays in 
the Geometric Invariant Theory of unipotent group actions on projective 
varieties, cf.~\cite[Section~5]{DorKir}.

\begin{ex}{\rm
Let $N$ act effectively on a smooth projective variety $X$. Any 
$N$-equivariant embedding $N \hookrightarrow \mathbb{P}(W)$, where $W$ is an 
$N$-representation space on which $N$ acts via an embedding $N \hookrightarrow 
{\rm{SL}}(W)$ is a Hamiltonian ${\rm{SL}}(W)$-extension. }
\end{ex}

\begin{ex}{\rm
Suppose that $G$ acts on $(X,\omega_X)$, extending the $N$-action. Since 
integration over $K$ does not change the cohomology class of $\omega_X$, we 
may, and will, suppose that $\omega_X$ is $K$-invariant. Therefore we can simply 
take $Z=X$ with $\iota=\id_X$ as Hamiltonian $G$-extension of the $N$-action on 
$X$.

On the other hand, the twisted product\footnote{The twisted product 
$G\times_NX$ is by definition the quotient of $G\times X$ by the proper 
holomorphic $N$-action given by $n\acts(g,x)=(gn^{-1},n\acts x)$. The $N$-orbit 
of $(g,x)$ will be denoted by $[g,x]\in G\times_NX$.} $G\times_NX$  is 
$G$-equivariantly isomorphic to $G/N\times X$ via the map 
$[g,x]\mapsto(gN,g\acts x)$ and hence embeds into $Z:=\ol{G/N}\times X$ where 
$\ol{G/N}$ is a smooth projective $G$-equivariant compactification of the 
quasi-affine homogeneous space $G/N$. Endowing $Z$ with a direct product 
K\"ahler metric $\omega_0\oplus\omega_X$ and considering $\iota\colon 
X\hookrightarrow Z$, $\iota(x)=(eN,x)$, we obtain another Hamiltonian 
$G$-extension of the $N$-action on $(X,\omega_X)$.}
\end{ex}

\begin{rem}{\rm
Although the automorphism group of a compact K\"ahler manifold has a natural 
structure of a meromorphic group acting meromorphically on $X$, 
see~\cite[Theorem~5.5]{Fuj}, this cannot be used in order to find a natural 
embedding of a unipotent algebraic group $N$ acting holomorphically on $X$ into 
a complex reductive group $G$ sitting inside $\Aut(X)$ as the following example 
shows.}
\end{rem}

\begin{ex}{\rm
Consider the connected algebraic group
\begin{equation*}
G=\left\{
\begin{pmatrix}
(ad-bc)^{-1} & z & w\\
0 & a & b\\
0 & c & d\\
\end{pmatrix};\ 
ad-bc\not=0\right\}\cong {\rm{GL}}(2,\mbb{C})\ltimes \mbb{C}^2.
\end{equation*}
According to~\cite[Theorem~1]{Br} there exists a $12$-dimensional 
projective 
complex manifold $X$ having $\Aut^0(X)$ isomorphic to $G$.

The group
\begin{equation*}
N=\left\{
\begin{pmatrix}
1&0&z\\
0&1&w\\
0&0&1\\
\end{pmatrix};\ z,w\in\mbb{C}\right\}\cong\mbb{C}^2
\end{equation*}
is a unipotent subgroup of $G$ which is not conjugate to a subgroup of the 
radical $R_u(G)$, nor to a subgroup of a Levi subgroup of $G$.

The group
\begin{equation*}
N=\left\{
\begin{pmatrix}
e^t&te^t&0\\
0&e^t&0\\
0&0&e^{-2t}\\
\end{pmatrix};\ t\in\mbb{C}\right\}\cong\mbb{C}
\end{equation*}
is a non-algebraic subgroup of $G$. Consequently, $N$ does not act 
meromorphically on $X$. Its Zariski closure is the group
\begin{equation*}
\ol{N}=\left\{
\begin{pmatrix}
t&s&0\\
0&t&0\\
0&0&t^{-2}\\
\end{pmatrix};\ t\in\mbb{C}^*, s\in\mbb{C}\right\}\cong\mbb{C}^*\times
\mbb{C}.
\end{equation*}
Note that $\ol{N}$ and hence $N$ are not conjugate to subgroups of neither a 
Levi subgroup nor the radical of $G$.}
\end{ex}

Let us state the main result of this section.

\begin{thm}\label{Thm:ExistenceExtensions}
Let $N$ be a unipotent subgroup of the simply-connected complex semisimple Lie 
group $G$. For a connected compact K\"ahler manifold $(X, \omega_X)$ endowed with a 
holomorphic $N$-action the following statements are equivalent.
\begin{enumerate}[\rm (1)]
\item There exists a Hamiltonian $G$-extension $(Z,\omega_Z)$ of the $N$-action 
on $X$.
\item The $N$-action on $X$ is meromorphic.
\item The twisted product $G\times_NX$ is K\"ahler.
\end{enumerate}
Moreover, given a meromorphic $N$-action on $X$ we can always find a 
Hamiltonian $G$-extension that is a $G$-equivariant compactification of 
$G\times_NX$.
\end{thm}

\begin{rem}{\rm
Since the centre of $G$ is finite and since $N$ does not contain finite 
subgroups, it follows that the $G$-action on $G\times_NX$ is effective whenever 
$N$ acts effectively on $X$.}
\end{rem}

Motivated by the above and by the equivalent conditions listed in 
Remark~\ref{rem:reductive_equivalence}, we make the following definition that is 
central to our discussion.

\begin{defn}{\rm
Let $N$ be a unipotent subgroup of the simply-connected complex semisimple Lie 
group $G$, and let $X$ be a connected compact K\"ahler manifold endowed with a 
holomorphic $N$-action. We say that $(X,\omega_X)$ is a 
\emph{Hamiltonian $N$-manifold} if there exists a Hamiltonian $G$-extension 
$(Z,\omega_Z)$ of the $N$-action on $X$. }
\end{defn}

The rest of this section is devoted to the proof of 
Theorem~\ref{Thm:ExistenceExtensions}.

\subsection{Necessary conditions for the existence of a $G$-extension}

In this section, we will show the implications ``(1) $\Rightarrow$ (3) 
$\Rightarrow$ (2)'' of Theorem~\ref{Thm:ExistenceExtensions}. Let $(X,\omega_X)$ 
be a connected compact K\"ahler manifold on which $N$ acts holomorphically, and 
let $\alpha\colon X \to \rm{Alb}(X)$ be the natural map from $X$ to its Albanese 
torus.

\begin{lem}[\protect{``(1) $\Rightarrow$ (3)''}]\label{Lem:necessarycond1}
If the $N$-action on $X$ admits a Hamiltonian $G$-extension, then the twisted 
product $G\times_NX$ is K\"ahler.
\end{lem}

\begin{proof}
Let us consider the induced proper embedding $Y=G\times_NX\hookrightarrow 
G\times_NZ\cong G/N \times Z$. The claim follows from the fact $G/N$ is 
quasi-affine, and hence K\"ahler.
\qed
\end{proof}

\begin{rem}
{\rm In fact, if $N$ is a connected nilpotent closed complex subgroup of $G$, then 
assuming that $G\times_N X$ be K\"ahler we can show that $N$ is in fact 
algebraic: applying \cite[Proposition II.1]{Bl} to the fibre bundle $G \times_N 
X \to G/N$, we see that $G/N$ is K\"ahler, an application of 
\cite[Corollary~4.12]{GMO} then yields the claim. This justifies our a priori 
algebraicity assumption on $N < G$.
}
\end{rem}

Next, we embark on proving the implication ``(3) $\Rightarrow$ (2)''. As a 
first step, we prove the following

\begin{lem}\label{Lem:necessarycond2}
If $G\times_N X$ is K\"ahler, then the action of $N$ on the Albanese torus 
$\Alb(X)$ is trivial.
\end{lem}

\begin{proof}
Suppose that the action of $N$ on $\Alb(X)$ is non-trivial. Then, as $N$ acts 
on $\Alb(X)$ by translations, we can find a closed one-parameter subgroup 
$\mbb{C}\hookrightarrow N$ such that $\dim\mbb{C}\acts[0]=1$ where $[0]$ denotes 
the base point of $\Alb(X)$. The topological closure 
$T:=\ol{\mbb{C}\acts[0]}\subset \alpha(X) \subset \Alb(X)$ is a connected 
compact subgroup of $\Alb(X)$, hence a subtorus. Its pre-image $\alpha^{-1}(T)$ 
is a compact $\mbb{C}$-invariant subvariety of $X$. According to the 
Jacobson-Morozov Theorem, see~\cite[Theorem~III.17]{Jac}, there is a closed 
subgroup $S$ of $G$ that is locally isomorphic to ${\rm{SL}}(2,\mbb{C})$ and 
contains the one-parameter subgroup $\mbb{C} \hookrightarrow N$ under 
discussion. Since $G\times_NX$ is K\"ahler and contains 
$Y':=S\times_\mbb{C}\alpha^{-1}(T)$ as a closed $S$-invariant analytic subset, 
every $S$-orbit in $Y'$ and therefore every $\mbb{C}$-orbit in $\alpha^{-1}(T)$ 
is Zariski-open in its closure by \cite[Corollary~3.9]{GMO}.

There exists a $1$-dimensional orbit $\mbb{C}\acts x$ in $X$ such that 
$\alpha(x) = [0] \in \Alb(X)$. Consequently, going through the possible 
stabiliser subgroups, we see that there are two possibilities: either the orbit 
$\mbb{C}\acts x$ is compact, i.e., it is a $1$-dimensional complex torus, or 
the normalisation of $\ol{\mbb{C}\acts x}$ is biholomorphic to $\mbb{P}_1$. 
Both cases lead to contradictions. Indeed, in the first case, the isotropy of 
$\mbb{C}$ and hence the isotropy in $S$ would be infinite discrete, which 
contradicts \cite[Proposition 4.4]{GMO}, while in the second case the fact that 
$\alpha|_{\ol{\mbb{C}\acts x}}$ is non-constant by construction would produce a 
non-zero holomorphic $1$-form on $\mathbb{P}^1$.
\qed
\end{proof}

We remark that the fact that $N$ acts trivially on $\Alb(X)$ alone does not 
imply the existence of a Hamiltonian $G$-extension of the $N$-action on $X$ as 
the following example shows.

\begin{ex}{\rm
Let $N=\left(\begin{smallmatrix}1&\mbb{C}\\0&1\end{smallmatrix}\right) \subset
G={\rm{SL}}(2,\mbb{C})$ and consider its action on $X=\mbb{P}_2$ given by
\begin{equation*}
t\acts[x_0:x_1:x_2]:=[e^tx_0:e^{it}x_1:x_2].
\end{equation*}
The $N$-orbits in $X$ are not locally Zariski closed, and there are elements 
having isotropy isomorphic to $\mbb{Z}$. According to~\cite[Theorem~3.6]{GMO} 
$G\times_NX$ cannot be K\"ahler. Hence, there does not exist a Hamiltonian 
$G$-extension of the $N$-action on $X$ although $N$ acts trivially on $\Alb(X) 
= \{\mathrm{pt}\}$.}
\end{ex}

Our next goal is to show that, if there exists a Hamiltonian $G$-extension of 
the $N$-action on $X$, then $N$ acts meromorphically on $X$. Due to 
Lemma~\ref{Lem:necessarycond2} the claim is equivalent to the fact that the map 
$N\to \Aut_{\rm{aff}}(X)$ induced by the action has analytically Zariski-closed 
image, where $\Aut_{\rm{aff}}(X)$ denotes the kernel of the Jacobi homomorphism
$\alpha_*\colon \Aut^0(X)\to\Aut^0(\Alb(X))$, cf.~\cite[\S2]{Fuj}. 
Since by \cite[Lemma 4.6 and Theorem 5.5]{Fuj} or \cite[Proposition 2.1]{Lie} 
the analytic Zariski-topology on the meromorphic subgroup $\Aut_{\rm{aff}}(X)  < 
\Aut^0(X)$ is obtained from the complex structure on the cycle space
$\mathscr{C}_{\dim X}(X\times X)$, as a first step we prove a technical lemma 
about induced actions on cycle spaces.

\begin{lem}\label{Lem:twistedBarletspace}
Let $N$ be a unipotent subgroup of $G$, let $M$ be a K\"ahler manifold endowed 
with a holomorphic $N$-action, and let $\mathscr{C}_k(M)$ be the Barlet space 
of compact $k$-cycles. Then, there exists a natural $G$-equivariant isomorphism
$G\times_N\mathscr{C}_k(M) \cong \mathscr{C}_k(G\times_NM)$.
\end{lem}

\begin{proof}
For every irreducible compact analytic subset $A$ of $M$ and each $g\in N$, the 
image $g(A)$ is again an irreducible compact analytic subset of $M$. Hence, we 
obtain a holomorphic $N$-action on $\mathscr{C}_k(M)$ by the obvious extension of this
action to $k$-cycles. Moreover, the inclusion $\iota\colon M\to G\times_NM$, 
$x\mapsto[e,x]$ induces a proper holomorphic embedding of $\mathscr{C}_k(M)$ 
into $\mathscr{C}_k(G\times_NM)$ by sending $A \in \mathscr{C}_k(M)$ to 
$\iota(A)$. From this we obtain a well-defined injective and immersive 
holomorphic map 
\begin{equation*}
\Phi_k\colon G\times_N\mathscr{C}_k(M)\hookrightarrow \mathscr{C}_k(G\times_NM),\quad
\Phi_k([g,A]):=g\acts\iota(A).
\end{equation*}

Suppose for a moment that $A\subset G\times_NM$ is an irreducible compact 
analytic subset. Since $G/N$ is quasi-affine, and hence in particular 
holomorphically separable, the bundle projection $\pi\colon G\times_NX\to G/N$ 
maps $A$ to a point $gN \in G/N$. Consequently, $\mathscr{C}_k(G\times_NM)$ 
coincides with the space $\mathscr{C}_k(G\times_NM)_\pi$ of $\pi$-relative 
cycles. Moreover, the natural $G$-action on $\mathscr{C}_k(G\times_NM)$ makes 
the resulting projection $\mathscr{C}_k(G\times_NM) \to G/N$ equivariant.

We hence conclude from the fact that $\Phi_k$ induces an isomorphism of the 
fibres over $eN$ that the image of $\Phi_k$ is all of 
$\mathscr{C}_k(G\times_NM)_\pi = \mathscr{C}_k(G\times_NM)$. The claim follows.
\qed
\end{proof}

\begin{prop}[``(3) $\Rightarrow$ (2)'']\label{prop:necessarycond3}
If $G\times_NX$ is K\"ahler, then $N$ acts meromorphically on $X$.
\end{prop}

\begin{proof}
Composed with the the natural embedding $\Aut_{\rm{aff}}(X) \hookrightarrow 
\mathscr{C}_n(X\times X)$, the action map $N \to \Aut_{\rm{aff}}(X)$ yields the 
following holomorphic map:
\begin{align*}
\iota\colon N&\to \mathscr{C}_n(X\times X)\\
g &\mapsto \Gamma_{x\mapsto g\acts x}:=\{(x,g\acts x)\mid x\in X\}.
\end{align*}
We set $n := \dim X$  and consider the $N$-action on the product $X\times X$ 
given by $g\acts(x,y)=(x,g\acts y)$. We have $G\times_N(X\times X)\cong 
(G\times_NX)\times X$, which shows that $G\times_N(X\times X)$ is K\"ahler. 
Hence, the cycle space $\mathscr{C}_n(G\times_N(X\times X))$ is 
K\"ahler by ~\cite[Th\'eor\`eme~2]{BV}. Applying 
Lemma~\ref{Lem:twistedBarletspace} to $M=N\times N$ we infer that 
$G\times_N\mathscr{C}_n(X\times X)$ is likewise K\"ahler. It therefore follows 
from~\cite[Theorem~3.6]{GMO} that all $N$-orbits in $\mathscr{C}_n(X\times X)$ 
are locally closed in the analytic Zariski-topology.

Now, we notice that $\iota(N)$ coincides with the $N$-orbit of $\Delta_X \in 
\mathscr{C}_n(X\times X)$, where $\Delta_X$ denotes the diagonal in 
$X\times X$. This implies that $\iota(N)$ is a locally Zariski-closed, hence 
Zariski-closed subgroup of $\Aut_{\rm{aff}}(X) < \Aut^0(X) \subset 
\mathscr{C}_n(X\times X)$, cf.~\cite[Section 7.4]{Humphreys}, as was to be 
shown.
\qed
\end{proof}

For later usage, we note the following, related observation.

\begin{lem}\label{lem:meromorphic_implies_Albanese_trivial}
An algebraic subgroup $H$ of $G$ which acts meromorphically on $X$ acts 
automatically trivially on $\Alb(X)$.
\end{lem}

\begin{proof}
This follows from the fact that every meromorphic homomorphism from an affine 
linear group to a compact complex torus is constant, 
see~\cite[Lemma~3.8]{Fuj}.
\qed
\end{proof}

\subsection{Existence of a Hamiltonian $G$-extension}

In this section, we will show the crucial implication ``(2) $\Rightarrow$ (1)'' 
of Theorem~\ref{Thm:ExistenceExtensions}, i.e., the fact that meromorphic 
actions of unipotent algebraic subgroups of semisimple groups $G$ always admit 
Hamiltonian $G$-extensions, thereby completing the proof. 

Let us recall the setup: Let $N$ be a unipotent subgroup of the simply-connected 
complex semisimple Lie group $G$, and let $X$ be a connected compact K\"ahler 
manifold endowed with a meromorphic $N$-action. 

\subsubsection{The case of unipotent radicals of Borel 
subgroups}\label{subsubsect:unipotent_radicals}

Let $B$ be a Borel subgroup of $G$ having Levi decomposition $B=TU = UT$. In a 
first step, we make the additional assumption that $N$ coincides with the the 
unipotent radical $U$ of $B$; i.e., $N=U$.

Let us consider the twisted product $M:=B\times_UX$ and the $U$-equivariant 
inclusion $\iota_X\colon X \hookrightarrow M$ as the fibre over $eU$. Since the 
principal bundle $B\to B/U\cong T$ is holomorphically trivial, the same holds 
for the associated fibre bundle; i.e., we have $M\cong T\times X$. Explicitly, 
an isomorphism $B\times_UX\to T\times X$ is given by $[tu,x]\mapsto(t,u\acts 
x)$. A direct calculation shows that the induced $B$-action on $T\times X$ is 
given by the formula
\begin{equation}\label{Eqn:B-action}
(tu)\acts(s,x)=(ts,(s^{-1}us)\acts x). 
\end{equation}

Let $\ol{G}$ be a projective meromorphic structure on $G$. Since the subgroups 
$T$, $U$ and $B$ are algebraic in $G$, their topological closures in $\ol{G}$ 
are analytic.

\begin{lem}\label{Lem:B-Compactification}
If $U$ acts meromorphically on $X$, then the $B$-action on $T\times X$ defined 
in Equation~\eqref{Eqn:B-action} extends to a meromorphic $B$-action on 
$\ol{T}\times X$. In particular, there exists a $B$-equivariant K\"ahler 
compactification $(\ol{M},\omega_{\ol{M}})$ of $M \supset \iota_X(X)$ such that 
$[\iota_X^*(\omega_{\ol{M}})] = [\omega_Z]$.
\end{lem}

\begin{proof}
By the definition of a meromorphic structure the map $c\colon T\times U\to U$, 
$(s,u)\mapsto s^{-1}us$, extends to a meromorphic map $\ol{c}\colon 
\ol{T}\times\ol{U}\to\ol{U}$ which is holomorphic on $(\ol{T}\times U) 
\cup(T\times\ol{U})$. Analogously, for the same reason the multiplication map 
$m\colon T \times T \to T$ extend meromorphically to $\ol{T} \times \ol{T} \to 
\ol{T}$ making $\ol{T}$ a $T$-bi-equivariant compactification of $T$. Denote the 
extended map by $\ol{m}$. With this notation, $B$ acts on $\ol{T}\times X$ by 
the formula
\begin{equation*}
(tu)\acts(s,x):= (\ol{m}(t,s), \ol{\alpha}(\ol{c}(s,u),x)),
\end{equation*}
where $\ol{\alpha}\colon\ol{U}\times X\to X$ is the meromorphic extension of the 
$U$-action on $X$. It is hence clear that $B$ acts meromorphically on $\ol{T}\times 
X$.
\qed
\end{proof}

In the next step we consider the holomorphic fibre bundle $Z:=G\times_B\ol{M} 
\overset{\pi}{\longrightarrow} G/B$. The natural inclusion $\ol{M} 
\hookrightarrow G\times_B \ol{M}$ as the fibre over $eB$ is denoted by 
$\iota_{\ol{M}}$. Both the typical fibre $\ol{M}$ and the base $G/B$ of $Z$ are 
K\"ahler. Using Blanchard's theorem~\cite[Th\'eor\`eme principal~II]{Bl}, we 
will construct a K\"ahler form on $Z$ such that $Z$ is a Hamiltonian 
$G$-extension of the $U$-action on $X$.

Since $G/B$ is simply-connected, in order to be able to apply Blanchard's 
result and hence show that $G\times_B\ol{M}$ is K\"ahler, we only have to check 
that the \emph{transgression map} from $H^1(\ol{M},\,\mbb{R})$ to 
$H^2(G/B,\,\mbb{R})$ is identically zero.  Recall that this map is 
defined on the set of transgressive elements in 
$H^1(\ol{M},\,\mbb{R})$ in the following way. A class $[\alpha]\in 
H^1(\ol{M},\,\mbb{R}\big)$ is called \emph{transgressive} if there exists a 
$1$-form $\beta$ on $Z$ such that $[\iota_{\ol{M}}^*\beta] =[\alpha]$ and such 
that there exists a $2$-form $\tau$ on $G/B$ such that $d\beta=\pi^*\tau$. Since 
$\pi^*$ is injective, we have $d\tau=0$, and the transgression map then 
associates $[\tau]\in H^2(G/B,\,\mbb{R})$ to $[\alpha]\in 
H^1(\ol{M},\,\mbb{R})$, see~\cite[Proposition 18.13]{BoTu}. As a 
side-remark we note that this transgression map is zero if and only if 
$b_1(Z)=b_1(G/B)+b_1(\ol{M})$ where $b_1$ denotes the first Betti number, 
see~\cite[p.~192]{Bl}.

\begin{lem}\label{lem:meromorphic_implies_transgressive}
Suppose that $B$ acts trivially on $\Alb(\ol{M})$. Then, every 
$[\alpha]\in H^1(\ol{M},\,\mbb{R})$ is transgressive, and the 
transgression map is identically zero.
\end{lem}

\begin{proof}
As $\ol{M}$ is K\"ahler by Lemma~\ref{Lem:B-Compactification}, we may use the 
Hodge decomposition to write
\begin{equation*}
[\alpha]\in H^1(\ol{M},\,\mbb{R})\subset 
H^1(\ol{M},\,\mbb{C})=H^{1,0}(\ol{M})\oplus 
\ol{H^{1,0}(\ol{M})} = H^0(\Omega^1_{\ol{M}}) \oplus 
\ol{H^0(\Omega^1_{\ol{M}})}
\end{equation*}
as $[\alpha]=\eta+\ol{\eta}$ where $\eta$ is a holomorphic $1$-form on $\ol{M}$. 
Since by hypothesis and Lemma~\ref{lem:meromorphic_implies_Albanese_trivial} the 
algebraic subgroup $B < G$ acts trivially on $\Alb(\ol{M})$, the form $\eta$ is 
$B$-invariant, hence extends to a $G$-invariant holomorphic $1$-form 
$\wh{\eta}\in H^0(\Omega^1_Z)^G$. Since $d\wh{\eta}$ is likewise 
$G$-invariant, it is uniquely determined by its restriction 
$\iota_{\ol{M}}^*(d\wh{\eta})=d(\iota_{\ol{M}}^*\wh{\eta})=d\eta=0$. 
Consequently, $d\wh{\eta}=0=\pi^*\bf{0}$. We conclude that $[\alpha]$ is 
transgressive and is mapped to $\mathbf{0} + \mathbf{0} = \mathbf{0} \in 
H^2(G/B,\, \mathbb{R} )$ by the transgression map.
\qed
\end{proof}

\begin{rem}{\rm
The proof works more generally for a \emph{parabolic} subgroup $P<G$ acting on a 
compact K\"ahler manifold such that the induced action on the Albanese is 
trivial.  To be more precise, we see that for every such $P$-manifold $X$ the 
twisted product $G\times_PX$ is K\"ahler. As the $G$-action on this manifold is 
Hamiltonian, by applying~\cite[Theorem~3.6]{GMO} or \cite[Remark after Lemma 
2.1]{Fuj2} together with \cite[Proposition 6.10]{Fuj} we conclude that every 
$P$-orbit in $X$ is Zariski-open in its closure and that the $P$-action on $X$ 
is meromorphic. This generalises and gives a new proof for a result of Sommese, 
cf.~\cite[\S3]{Som}. For criteria guaranteeing triviality of induced actions on 
Albanese tori see \cite[\S6c)]{Fuj}.}
\end{rem}

Combining Lemmata~ \ref{Lem:B-Compactification}, 
\ref{lem:meromorphic_implies_Albanese_trivial}, and 
\ref{lem:meromorphic_implies_transgressive} with Blanchard's theorem we conclude 
that the twisted product $Z=G\times_B\ol{M}$ is K\"ahler. Moreover, the first 
step of Blanchard's proof~\cite[p.~187]{Bl} shows that the cohomology class of 
the constructed K\"ahler form on $Z=G\times_B\ol{M}$ pulls back under 
$\iota_{\ol{M}}$ to $[\omega_{\ol{M}}]$ on $\ol{M}$. Embedding $X$ into $\ol{M}$ 
by $\iota_{X}$ and further into $Z$ by $\iota_{\ol{M}}$ we obtain a Hamiltonian 
$G$-extension of the $U$-action on $X$. Notice that $Z$ contains 
$G\times_UX\cong G\times_B(B\times_UX)$ as a Zariski open subset. In fact, $Z = 
G \times_B \ol{M} \cong G \times_B( \ol{T} \times X ) \to G \times_B \ol{T}$ is 
an extension of the $X$-fibre bundle $G\times_U X \to G/N$ to the projective 
completion $G \times_B \ol{T}$ of $G/N$.

\subsubsection{The general case}

In the second and final step we show that the case of an arbitrary unipotent 
subgroup $N$ of $G$ can be reduced to the unipotent radical $U$ of a Borel 
subgroup $B\subset G$.

\begin{lem}\label{Lem:U-Compactification}
Let $N$ be a unipotent subgroup of $G$. Then there exists a Borel subgroup 
$B=TU$ of $G$ such that $N\subset U$. In addition, if $N$ acts effectively and 
meromorphically on $X$, then $U\times_NX$ admits a $U$-equivariant K\"ahler 
compactification on which $U$ acts effectively and meromorphically and whose 
K\"ahler class extends the given K\"ahler class $[\omega_X]$ on $X$.
\end{lem}

\begin{proof}
The first statement is \cite[30.4, Theorem]{Humphreys}. Moreover, from the 
first paragraph in the proof of \cite[Chapter 5, Theorem 4]{Akhiezer} we 
conclude that the $N$-principal bundle $U \to U/N$ admits an algebraic section, 
whose image we call $S$. As a consequence, we obtain an $N$-equivariant 
isomorphism $N \times S\to U$ with $S \cong U/N$. Let $\ol{S}$ be a 
$U$-equivariant smooth projective compactification of $S$. Following the line of 
argumentation presented at the beginning of 
Section~\ref{subsubsect:unipotent_radicals} we conclude as before that 
$\ol{S}\times X$ is a compact K\"ahler manifold endowed with a meromorphic 
$U$-action, containing $U \times_NX$ as a Zariski-open set. Again, we can choose 
the K\"ahler form on $\ol{S}\times X$ such that its class extends the given 
K\"ahler class $[\omega_X]$.
\qed
\end{proof}

Finally, by applying the discussion of 
Section~\ref{subsubsect:unipotent_radicals} to the compact K\"ahler manifold $X' 
= \ol{S} \times X$ with meromorphic $U$-action obtained in 
Lemma~\ref{Lem:U-Compactification} and noticing that 
\begin{equation}\label{eq:twisttwist}
Y=G\times_NX\cong G\times_B(B\times_U(U\times_NX))\cong
G\times_B( T\times (U/N) \times X)
\end{equation}
we arrive at

\begin{prop}[``(2) $\Rightarrow$ (1)'']\label{Prop:construction}
Let $N$ be a unipotent subgroup of $G$ acting effectively and meromorphically 
on the compact K\"ahler manifold $X$. Then, there exists a Hamiltonian 
$G$-extension $Z$ of the $N$-action on $X$, which can be chosen such that it 
contains $Y:=G\times_NX$ as a Zariski open subset. More precisely, we might 
choose
\begin{equation*}
Z = \ol{Y}= G\times_B( \ol{T}\times \ol{U/N} \times X),
\end{equation*}
where $B=TU$ is a Borel subgroup of $G$ with $N\subset U$ and where 
$\ol{T}\times \ol{U/N}$ is the $B$-equivariant compactification of 
$T\times(U/N)$ constructed in Lemmata~\ref{Lem:B-Compactification} 
and~\ref{Lem:U-Compactification}.
\end{prop}

\begin{rem}[Non-uniqueness]{\rm
There are many possible ways to choose the compactifications
$\ol{T}$ and 
$\ol{S}$; in the first case, we have all smooth projective toric varieties of 
dimension $\dim T$ to choose from.}
\end{rem}

\subsubsection{Additional observations}

We note a minimality property of the above construction that is crucial for 
subsequent arguments on the independence of the set of semistable points from 
the chosen $G$-extensions, see Section~\ref{subsect:semistability_independent} 
below. 

\begin{prop}\label{Prop:Minimality}
Let $X$ be a compact Hamiltonian $N$-manifold and let $\ol{Y}$ be the 
$G$-equivariant compactification of $Y:=G\times_NX$ whose existence is 
guaranteed by Proposition~\ref{Prop:construction}. Then, for any Hamiltonian 
$G$-extension $Z$ of the $N$-action on $X$, there exists a $G$-equivariant 
embedding $\ol{Y}\hookrightarrow\ol{G/N}\times Z$ where $\ol{G/N}$ is a certain 
$G$-equivariant compactification of $G/N$.
\end{prop}

\begin{proof}
Let $Z$ be any $G$-extension of the $N$-action on $X$. The $N$-equivariant 
embedding $X\hookrightarrow Z$ induces a $G$-equivariant embedding 
$Y=G\times_NX\hookrightarrow G\times_NZ\cong G/N \times Z$. The identification 
$G/N\cong G\times_B (T\times U/N )$, cf.~Equation~\eqref{eq:twisttwist}, 
suggests choosing the compactified fibre bundle 
\begin{equation}\label{eq:compactify_GmodN}
 \ol{G/N}:=G\times_B(\ol{T}\times\ol{U/N})
\end{equation}
as a well-adapted $G$-equivariant compactification of $G/N$. With this 
definition we obtain the desired embedding $\ol{Y}=G\times_B( \ol{T}\times 
\ol{U/N} \times X)\hookrightarrow G\times_B( \ol{T}\times \ol{U/N} 
\times Z)\cong \ol{G/N}\times Z$.
\qed
\end{proof}

\begin{rem}\label{Rmk:simplyconnected}{\rm
Since $N$ is a connected subgroup of the simply-connected group $G$, the 
quasi-affine variety $G/N$ is simply-connected. Consequently, every smooth 
compactification of $G/N$ is likewise simply-connected, see~\cite[0.7(B)]{FL}; 
in particular, this observation applies to the compactification constructed in 
Proposition~\ref{Prop:Minimality}.}
\end{rem}

\begin{ex}{\rm
Suppose that $N=\left(\begin{smallmatrix}1&\mbb{C}\\0&1\end{smallmatrix}
\right)\subset G={\rm{SL}}(2,\mbb{C})$. Choosing $\ol{T}=\mbb{P}_1$, we see that 
$\ol{G/N}=G\times_B\ol{T}$ is the blow-up of $\mbb{P}_2$ at the origin, i.e., 
the first Hirzebruch surface.}
\end{ex}

\section{The set of semistable points with respect to unipotent 
groups}\label{Section:Semistable}

Let $G=K^\mbb{C}$ be a simply-connected semisimple complex Lie 
group with maximal compact subgroup $K$, let $N$ be a unipotent subgroup of $G$, 
and let $(X,\omega_X)$ be a compact Hamiltonian $N$-manifold. In this section we 
explain how one can use the concept of Hamiltonian $G$-extensions in order to 
define the set of semistable points $X^{ss}_N[\omega_X]$ for the $N$-action on 
$X$. 

\subsection{Defining the set of semistable points}

We will slowly approach the goal of defining the correct notion of 
semistability, see Definition~\ref{defn:semistable} below.

\subsubsection{The naive approach}\label{subsubsect:naive_approach}
Since $(X,\omega_X)$ is a compact Hamiltonian $N$-manifold, we can find a 
Hamiltonian $G$-extension $(Z,\omega_Z)$ of the $N$-action on $X$ with 
$N$-equivariant embedding $\iota\colon X\hookrightarrow Z$. It is tempting to 
define the set of $N$-semistable points in $X$ as the $N$-invariant open subset 
$\iota^{-1}(Z^{ss}_G[\omega_Z])$. As the following example shows, this set 
heavily depends on the choice of the $G$-extension $(Z,\omega_Z)$.

\begin{ex}{\rm
Let $X=\mbb{P}_1$ with the Fubini-Study metric $\omega_{FS}$ and let $N=\mbb{C}$ 
act on $\mbb{P}_1$ by $t\acts[x_0:x_1]=[x_0+tx_1:x_1]$. Since the $N$-action 
extends to an action of $G={\rm{SL}}(2,\mbb{C})$, we have the following three 
Hamiltonian $G$-extensions.
\begin{enumerate}[\rm (1)]
\item $Z_1=X$ where $\iota_1=\id_X$,
\item $Z_2=X\times X$ endowed with $\omega_{Z_2}=\frac{1}{2}(\omega_{FS}\oplus
\omega_{FS})$ and the diagonal $G$-action, where the embedding is given by 
$\iota_2(x)=(x,x)$, and
\item $Z_3=Z_2$ with $\omega_{Z_3}=2\omega_{Z_2}$ and with the diagonal 
$G$-action where $\iota_3(x)=([1:0],x)$.
\end{enumerate}
It is not hard to show that $\iota_1^{-1}(Z_{1,G}^{ss})=\iota_2^{-1}
(Z_{2,G}^{ss})=\emptyset$ while $\iota_3^{-1}(Z_{3,G}^{ss})=
\mbb{P}_1\setminus \{[1:0]\}$.}
\end{ex}

\subsubsection{The case of linear actions on projective 
manifolds}\label{subsubsect:algebraic_setup}

Let us recall the approach taken by Doran and Kirwan in \cite{DorKir}. There, 
the authors study the situation where $X$ is a subvariety of $\mbb{P}(W)$ for 
some finite-dimensional $G$-module $W$ and they consider the embedding 
$\wh{\iota}\colon G\times_NX\hookrightarrow 
G\times_N\mbb{P}(W)\cong(G/N)\times\mbb{P}(W)$ that is induced by the embedding 
$\iota\colon X\hookrightarrow\mbb{P}(W)$. Then, they equip $G\times_NX$ with the 
$G$-linearised ample line bundle 
$L:=\wh{\iota}^*(\mathrm{pr}_1^*(\mathscr{O}_{G/N})\otimes 
\mathrm{pr}_2^*(\mathscr{O}_{\mbb{P}(W)}(1)) )$. In this way, Doran and 
Kirwan can consider the set of Mumford-semistable points 
$(G\times_NX)^{ss}_G(L)$ in $G\times_NX$ and they proceed by defining the set of
Mumford-semistable points in $X$ as $X^{ss}_N(L):=\iota^{-1}(G\times_NX)^{ss}_G(L)$,
see~\cite[Definition~5.1.5]{DorKir}. As shown in~\cite[Propositions~5.1.8 
and~5.1.9]{DorKir}, this set $X^{ss}_N$ does not depend on the choice of the 
group $G$ and can be intrinsically defined knowing only the $N$-action on $X$.

\subsubsection{The definition of semistability}\label{subsubsectio:def_of_ss}

The algebraic approach suggests that in our situation, we should choose a 
Hamiltonian $G$-extension $(Z,\omega_Z)$ of the $N$-action on $X$ and then 
consider the analogous embedding $Y:=G\times_NX\hookrightarrow G\times_NZ\cong 
G/N \times Z$, such that the Hamiltonian $G$-extension $Z$ plays the role of the 
projectivised $G$-module $\mbb{P}(W)$.

To implement this idea, first note that the $N$-equivariant embedding 
$\iota\colon X\hookrightarrow Z$ induces the $G$-equivariant proper embedding
\begin{equation}\label{eq:untwisting_the_action}
\wh{\iota}\colon Y=G\times_NX\hookrightarrow G\times_NZ\cong G/N\times Z
\end{equation}
given by $\wh{\iota}([g,x])=(gN,g\acts\iota(x))$. Secondly, recalling 
that $G/N$ is quasi-affine, choose a $G$-representation space $V$ containing 
$G/N$ as an orbit as well as a  $K$-invariant hermitian inner product on $V$, 
which induces $K$-invariant K\"ahler forms $\omega_V$ on $G/N$ and 
\begin{equation}\label{eq:definition_of_omega_1}
\wh{\omega}_Z:= \omega_V\oplus 
\omega_Z
\end{equation}
on $G/N \times Z$. Then, we take
\begin{equation}\label{eq:definition_of_omega_2}
 \omega_Y:=\wh{\iota}^*\wh{\omega}_{Z}
\end{equation}
as K\"ahler form on $Y$.  The spaces considered so far fit into  following 
commutative diagram
\begin{equation*}
\xymatrix{
X\ar@{^(->}[r]^{\iota}\ar[d] & Z\ar[d] \\
Y=G\times_NX\ar@{^(->}[r]^{\wh{\iota}} & (G/N)\times Z,}
\end{equation*}
where the vertical arrows correspond to the maps given by $x\mapsto [e,x]$ and 
$z\mapsto(eN,z)$, respectively. It follows that $[\omega_Y]$ extends $[\omega_X]$ 
to $Y$. Recall that, as $K$ is semisimple and $\omega_Y$ is $K$-invariant, 
there exists a unique moment map $\mu_Y\colon Y \to \mathfrak{k}^*$ for the 
$K$-action on $Y$, see Remark~\ref{rem:semisimple_always_Hamiltonian}; in other 
words, $Y$ is a (non-compact) Hamiltonian $G$-manifold. We will use $\mu_Y$ to 
define a notion of semistability for the $N$-action:

\begin{defn}\label{defn:semistable}{\rm
Let $(X,\omega_X)$ be a Hamiltonian $N$-manifold. We define the \emph{set of 
$N$-semistable points (with respect to $[\omega_X]$)} as
\[\boxed{X^{ss}_N[\omega_X]:=X\cap Y^{ss}_G[\omega_Y]} \;,  
\]
where the K\"ahler form $\omega_Y$ on $Y=G\times_NX$ is given by Equations 
\eqref{eq:definition_of_omega_1} and \eqref{eq:definition_of_omega_2} above. 
Analogously, the \emph{set of $N$-stable points} in $X$ is defined as 
$X^s_N[\omega_X]:=X\cap Y^s_G[\omega_Y]$, where we set
$Y^s_G[\omega_Y]:=G\acts\mu_Y^{-1}(0)$.\footnote{With this definition, a semistable point $x\in X$ is $N$-stable if and only if
$N\acts x$ is closed in $X^{ss}_N[\omega_X]$. In view of Lemma~\ref{Lem:AllOrbitsClosed} and Theorem~\ref{thm:existence_of_geometric_quotient} below this is
the correct notion in our situation. The reader should however be aware that 
there are several notions of (proper) stability in use in the literature.
}}
\end{defn}

Note that $X^{ss}_N[\omega_X]$ a priori depends on the choice of $G$, 
$\omega_V$, and $Z$ although we do not convey this information in our notation. 
We will discuss the choices regarding $G$ and the metric on $G/N$ in 
Section~\ref{subsect:discussion_of_choice} below. As we will see, the definition 
is actually independent of the choice of a Hamiltonian $G$-extension, once $G$ 
and the metric on $G/N$ are fixed; see the subsequent section.

\begin{ex}[A non-projective compact K\"ahler manifold with meromorphic 
$\mathbb{C}$-action]{\rm
Let $S$ be a non-projective $K3$-surface with $\mathrm{Pic}(S) \neq \{e\}$, 
and let $L \to S$ be a non-trivial holomorphic line bundle on $S$ with zero 
section $Z_L \subset L$. Let $P:= L \setminus Z_L \to S$ be the associated 
$\mathbb{C}^*$-principal bundle, and consider the non-trivial 
$\mathbb{P}^1$-bundle \[M:= P \times_{\mathbb{C}^*} \mathbb{P}^1 
\longrightarrow P/\mathbb{C}^* = S.\] As the Albanese of $\mathbb{P}^1$ is 
trivial and since $S$ is a simply-connected compact K\"ahler manifold, the compact 
complex manifold $M$ is K\"ahler by Blanchard's theorem, cf.~the discussion in 
Section~\ref{subsubsect:unipotent_radicals}. On the other hand, as the 
non-projective surface $S$ embeds into $M$ as zero section, $Y$ is likewise 
non-projective. Since the corresponding vector field has zeroes, the effective
$\mathbb{C}^*$-action from the left is trivial on the Albanese of $M$, see 
\cite[Prop.~6.8]{Fuj} or \cite[Proposition 3.14]{Lie}.  Trivially extend the 
action of $\mathbb{C}^*$ on $M$ to an action of the Borel subgroup 
$\mathbb{C}^* < B$ of lower-triangular matrices in $\mathrm{SL}(2,\mathbb{C}) =: 
G$, and define \[X := G \times_B M.\] Since the $B$-action on the Albanese of 
$M$ is trivial and since $\mathbb{P}^1$ is simply-connected, it again follows 
from Blanchard's theorem that the compact complex manifold $X$ is K\"ahler, see 
also Lemma~\ref{lem:meromorphic_implies_transgressive}. Moreover, the 
$G$-action on $X$ induces an action of the unipotent algebraic subgroup $N < 
G$ of strictly upper-triangular elements of $\mathrm{SL}(2, \mathbb{C})$. Set 
$K:= {\rm{SU}}(2)$. Then, any $K$-invariant K\"ahler form $\omega_X$ on $X$ produces a 
moment map $\mu_X\colon X \to \mathfrak{k}^*$ for the $K$-action; i.e., $X$ is 
its own Hamiltonian $G$-extension, so that the $N$-action on $X$ is meromorphic.

Now, apply the construction described at the beginning of the present section 
to $X$. As $G> N$ already acts on $X$, we can choose $Z=X$ as Hamiltonian 
$G$-extension. Let $\omega_V$ be the essentially unique flat $K$-invariant 
K\"ahler form on $G/N \cong \mathbb{C}^2 \setminus\{0\} \subset V=\mathbb{C}^2$, 
and hence consider the K\"ahler form $\omega_Y = \omega_V \oplus \omega_X$ on $Y 
:= G \times_N X \cong G/N \times X$, together with the resulting moment map 
$\mu_Y = \mu_{G/N} + \mu_X\colon Y \to \mathfrak{k}^*$. We claim that 
$X^{ss}_N[\omega_X] \neq \emptyset$, which will conclude our construction. In 
order to establish the claim, we first note that $\mu_X(X) \neq \{0\}$, 
as otherwise Equation \eqref{eq:magic_formula} would imply that 
the $K$-action on $X$ is trivial. Let $\beta_0 \in \mu_X(X) \setminus \{0\}$ 
and choose $x_0 \in \mu_X^{-1}(\beta_0)$. We have $\mu_{G/N}(G/N) = 
\mathfrak{k}^* \setminus \{0\}$, cf.~Section~\ref{sect:moment_maps_for_reps}, 
and hence there exists $g_0 \in G$ such that $\mu_{G/N}(g_0U) = -\beta_0$.  By 
construction, we then have $\mu_Y(g_0\acts(eU, g_0^{-1}\acts x_0)) 
=\mu_Y(g_0U, x_0) = 0$, and hence $(eU, g_0^{-1}\acts x_0) \in 
Y^{ss}_G[\omega_Y] \cap (\{eU\} \times X) = X^{ss}_N[\omega_X]$, which is 
therefore non-empty, as claimed.}
\end{ex}

\subsection{Semistability does not depend on the 
$G$-extension}\label{subsect:semistability_independent}

In contrast to the naive definition of semistability discussed in 
Section~\ref{subsubsect:naive_approach} above, it turns out the set of 
$G$-semistable points in $Y$ with respect to $\omega_Y$ as defined in 
Definition~\ref{defn:semistable} does not depend on the choice of the 
Hamiltonian $G$-extension 
$(Z,\omega_Z)$.

\begin{thm}\label{Thm:Independence}
Let $N$ be a unipotent subgroup of the simply-connected
semisimple complex Lie group 
$G$, acting meromorphically on the compact K\"ahler manifold $(X,\omega_X)$. 
Let $(Z_j,\omega_j)$, $j=1,2$, be two Hamiltonian $G$-extensions of the 
$N$-action on $X$. Choose a $G$-equivariant algebraic embedding $G/N
\hookrightarrow V$ into a $G$-representation space $V$, and a $K$-invariant 
Hermitian inner product on $V$, inducing K\"ahler forms $\omega_V$ on $G/N$, 
$\wh{\omega}_Z$ on $G/N \times Z$, and $\omega_{Y,1}$, $\omega_{Y,2}$ on $Y$, as 
described in Section~\ref{subsubsectio:def_of_ss}. Then, we have 
$Y^{ss}_G[\omega_{Y,1}]=Y^{ss}_G[\omega_{Y,2}]$.
\end{thm}

\begin{proof}
We will prove that on the smooth K\"ahler compactification $\ol{Y}$ of $Y$ 
constructed in Proposition~\ref{Prop:construction} there exists a $(1,1)$-form 
$\alpha\in\mathcal{A}^{1,1}(\ol{Y})$ satisfying the following two properties: 
\begin{enumerate}
 \item[(a)] $\alpha|_Y = \omega_{Y,2}-\omega_{Y,1}$, and
 \item[(b)] $[\alpha]=0\in H^2(\ol{Y},\,\mbb{R})$.
\end{enumerate}
It then follows from the $\partial\ol{\partial}$-lemma  on the compact K\"ahler 
manifold $\ol{Y}$ that there is a smooth function $\varphi \in 
\mathcal{C}^\infty(\ol{Y})$  such that $\alpha=i\partial\ol{\partial}\varphi$. 
Restricting everything to $Y$, we obtain a \emph{bounded} smooth function 
$\varphi$ on $Y$ such that 
$\omega_{Y,2}=\omega_{Y,1}+i\partial\ol{\partial}\varphi$. In this situation, we 
can repeat the proof of Proposition~\ref{Prop:Semistable} to deduce 
$Y^{ss}_G[\omega_{Y,1}]=Y^{ss}_G[\omega_{Y,2}]$. Hence, in order to complete the 
argument, we must show existence of $\ol{Y}$ and $\alpha$ with the above 
properties.

To do so, we note first that $(Z,\omega_Z):=(Z_1\times Z_2, \frac{1}{2}(
\omega_1\oplus\omega_2))$ is another Hamiltonian $G$-extension of the 
$N$-action on $X$; here, $G$ acts diagonally on $Z_1\times Z_2$ and 
$\iota\colon X\hookrightarrow Z$ is given by the direct product 
$\iota(x)=(\iota_1(x),\iota_2(x))$ of the two inclusions 
$\iota_j\colon X \hookrightarrow Z_j$, $j=1,2$. Our situation can be summarised 
by the following diagram:
\[\begin{xymatrix}{
  & & G/N \times Z_1 \ar[r]   &  Z_1 \\
 Y \ar[rr]^{\wh{\iota}} \ar[rru]^{\wh{\iota}_1}    \ar[rrd]^{\wh{\iota}_2}    & 
    &    G/N \times Z  \ar[r]^<<<<{\mathrm{pr}_Z} \ar[u]_{q_1}  \ar[d]_{q_2}& Z 
\ar[u] \ar[d] \\
  &  &   G/N \times Z_2 \ar[r]  & Z_2.
}
  \end{xymatrix}
\]
We denote the K\"ahler form $\omega_V \oplus \omega_j$ on $G/N \times Z_j$ by 
$\wh{\omega}_j$, see Equation~\eqref{eq:definition_of_omega_1}, and note that by 
assumption the same form $\omega_V$ appears in both formulas.  Then, it follows 
from the general construction and from the diagram above that
\begin{equation}\label{eq:compute_difference}\omega_{Y,2} - \omega_{Y,1} = 
\wh{\iota}^*(q_2^*(\wh{\omega}_2) - q_1^*(\wh{\omega}_1)) = 
\wh{\iota}^* (\mathrm{pr}_Z^*(\omega_2 - \omega_1) ).
\end{equation} 
Let $\ol{G/N}$ as defined in Equation~\eqref{eq:compactify_GmodN} and let us 
denote the natural projection $\ol{G/N} \times Z \to Z$ by 
$\ol{\mathrm{pr}}_Z$. As $Z$ is a Hamiltonian $G$-extension, by 
Proposition~\ref{Prop:Minimality} there exists a $G$-equivariant holomorphic 
embedding $\psi\colon \ol{Y} \hookrightarrow \ol{G/N} \times Z$. It now follows 
from Equation~\eqref{eq:compute_difference} that 
\begin{equation}\label{eq:def_of_alpha}
\alpha:= \psi^*(\ol{\mathrm{pr}}_Z^*(\omega_2 - \omega_1)) \in 
\mathcal{A}^{1,1}(\ol{Y})
\end{equation}
fulfils property (a), as desired.

We still must show that the $K$-invariant $2$-form $\alpha$ is cohomologuous to 
zero on $\ol{Y}$. For this, we consider the holomorphic fibre bundle $q := 
\mathrm{pr}_{\ol{G/N}} \circ \psi \colon \ol{Y} \hookrightarrow \ol{G/N} \times 
Z \to \ol{G/N}$ with typical fibre $X$ and base $\ol{G/N}$. Since $\ol{Y}$ is 
K\"ahler, and as by Remark~\ref{Rmk:simplyconnected} the manifold $\ol{G/N}$ is 
simply-connected, it follows from \cite[Th\'eor\`eme II.1.1]{Bl} that the 
transgression map $H^1(X,\, \mathbb{R}) \to H^2(\ol{G/N},\, 
\mathbb{R})$ is zero. Consequently, the Leray spectral sequence for $q$ 
degenerates at the $E_2$-term by \cite[Th\'eor\`eme II.1.2]{Bl}; see also 
\cite[Theorem~4.15 and Remark 4.16]{VoisinII}. Again using simple-connectedness 
of $\ol{G/N}$, we conclude that 
\[ 
H^{k}(\ol{G/N},\, \mathbb{R} ) \otimes H^{l}(X,\, \mathbb{R} 
) = E_2^{k,l} = E_\infty^{k,l}= \mathrm{Gr}^kH^{k+l}(\ol{Y},\, 
\mathbb{R}) \quad\; \text{  and } \quad\;  E^{1,1}_\infty = \{0\}.
\] 
Computing the corresponding filtration of $H^2( \ol{Y}, \, \mathbb{R} 
)$ and comparing it with the Leray spectral sequence for 
$\mathrm{pr}_{\ol{G/N}}$ then leads to the following commutative diagram
\[
\begin{xymatrix}{ 0 \ar[r]& H^2(\ol{G/N},\, \mathbb{R}) \ar[r]^{q^*}& 
H^2(\ol{Y}, \,\mathbb{R} )\ar[r]^{j_X^*} & H^2(X,\, \mathbb{R} 
) \ar[r]& 0\\
  0 \ar[r]& H^2(\ol{G/N},\, \mathbb{R}) \ar[u]^= 
\ar[r]^<<<<{\mathrm{pr}_{\ol{G/N}}^*}&  H^2(\ol{G/N},\, \mathbb{R}) \oplus 
H^2(Z,\, \mathbb{R}) \ar[u]^{\psi^*}  \ar[r]^<<<<{j_Z^*} & 
H^2(Z,\, \mathbb{R} )  \ar[u]^{\iota^*}\ar[r]& 0}
\end{xymatrix}
\]
where $j_X\colon X \hookrightarrow \ol{Y}$ is the inclusion as the fibre over 
$eN \in G/N$, and similarly for $j_Z$.

Now, Equation~\eqref{eq:def_of_alpha} says that $[\alpha] = \psi^*([0] \oplus 
[\omega_2-\omega_1])$. Together with $[j_X^*(\alpha)] = [\iota_2^*(\omega_2)] - 
[\iota_1^*(\omega_1)] = [\omega_X]- [\omega_X] = 0$ this implies that $[\alpha] 
= 0$, as claimed.
\qed
\end{proof}

\begin{rem}{\rm
Note that in contrast to their difference the forms $\omega_{Y,j}$ themselves 
\emph{do not extend} to the compactification $\ol{Y}$.}
\end{rem}

\subsection{Discussion regarding the choice of K\"ahler metric on 
$G/N$}\label{subsect:discussion_of_choice}

As this is a subtle issue, let us discuss the choice of K\"ahler forms on $G/N$ 
made in Section~\ref{subsubsectio:def_of_ss} and the fact that we have to fix 
such a form in some detail. We will provide examples showing that the 
independence statement of Theorem~\ref{Thm:Independence} is optimal from many 
points of view. The problems occuring are closely related to the ones 
encountered in the algebraic situation when searching for various kinds of 
``reductive envelopes'', cf.~\cite[Sections~5.2 and 5.3]{DorKir}. 

\subsubsection{The algebraic situation}

Let us compare with Doran--Kirwan's approach in the algebraic situation, see 
Section~\ref{subsubsect:algebraic_setup}: The key point that explains the choice 
of the trivial line bundle on $G/N$ and that eventually makes the proof of 
\cite[Proposition~5.1.9]{DorKir} on the independence of semistability from the 
choice of the embedding into $G$ work is the following. Given any pair $G$ and 
$G'$ of reductive groups such that $N\subset G\subset G'$, the line bundles $L$ 
on $G\times_NX$ and $L'$ on $G'\times_NX$ constructed in~\cite{DorKir} as above 
verify $\iota^*L'=L$ where $\iota\colon G\times_NX\hookrightarrow G'\times_NX$ 
is the embedding induced by the inclusion $G\hookrightarrow G'$. In the analytic 
category, such a canonical choice of K\"ahler metric on $G/N$ does not exist, 
even among curvature forms in the trivial line bundle. Indeed, every 
$K$-invariant K\"ahler metric of the form $\omega=i\partial\ol{\partial}\rho$ 
with $\rho\in\mathcal{C}^\infty(G/N)^K$ is the curvature form of a $K$-invariant 
hermitian metric in the trivial line bundle on $G/N$, cf.~\cite[Chapter V, 
(12.6)]{Dem}. Even if we restrict to metrics of this form, the set of semistable 
points might change, as the following example shows.

\begin{ex}\label{ex:varying_c}{\rm
Let us consider the algebraic and hence meromorphic action of $\mbb{C}\cong 
N\subset G={\rm{SL}}(2,\mbb{C})$ on $X=\mbb{P}_1$, endowed with the Fubini-Study 
form $\omega_{FS}$. As Hamiltonian $G$-extension of the $N$-action on $X$ we 
take $Z=X$. Let $K={\rm{SU}}(2)$. 

According to~\cite[Lemma~7.10]{Dem}, every $K$-invariant K\"ahler form on 
$G/N\cong\mbb{C}^2\setminus\{0\}$ is of the form 
$i\partial\ol{\partial}\rho(\norm{z})$ where $\rho$ is a smooth 
function on $\mbb{R}^{>0}$ such that $\rho\circ\exp$ is strictly increasing 
and strictly convex. Let $\varphi\in\mathcal{C}^\infty(\mbb{R}^{>0})$ be the 
function defined by $\varphi(t^2)=\rho(t)$. Then the unique moment map for the 
action of ${\rm{SU}}(2)$ on $\mbb{C}^2\setminus\{0\}$ is given by
\begin{equation}\label{eq:general_moment_map}
z=(z_1,z_2)\mapsto \varphi'(\norm{z}^2)
\begin{pmatrix}
\frac{\abs{z_1}^2-\abs{z_2}^2}{2}&\ol{z_1}z_2\\
z_1\ol{z_2}&-\frac{\abs{z_1}^2-\abs{z_2}^2}{2}
\end{pmatrix}.
\end{equation}

We now consider the one-parameter family of K\"ahler forms $\omega_{G/N,c}$ 
given by $\omega_{G/N,c} := i\partial\ol{\partial}\rho_c(\norm{z})$, where 
$\rho_c(t)=c\log(1+t^2)$ with $c>0$. Following the construction of 
Section~\ref{subsubsectio:def_of_ss}, the induced K\"ahler form on $Y = G/N 
\times\mbb{P}_1$ is 
$\omega_c=i\partial\ol{\partial}\rho_c(\norm{z})\oplus\omega_{FS}$. Identifying 
$\lie{su}(2)^*$ with $i\lie{su}(2)$ via the Killing form, the corresponding 
moment map $\mu\colon (G/N)\times\mbb{P}_1\to i\lie{su}(2)$ is given by
\begin{equation*}
\mu(z,[x_0:x_1])=
\frac{c}{1+\norm{z}^2}
\begin{pmatrix}
\frac{\abs{z_1}^2-\abs{z_2}^2}{2}&\ol{z_1}z_2\\
z_1\ol{z_2}&-\frac{\abs{z_1}^2-\abs{z_2}^2}{2}
\end{pmatrix}
+\frac{1}{\abs{x_0}^2+\abs{x_1}^2}
\begin{pmatrix}
\frac{\abs{x_0}^2-\abs{x_1}^2}{2}&\ol{x_0}x_1\\
x_0\ol{x_1}&-\frac{\abs{x_0}^2-\abs{x_1}^2}{2}
\end{pmatrix}.
\end{equation*}
A slice for the ${\rm{SU}}(2)$-action on $(G/N)\times\mbb{P}_1$ is given by 
$S=\{((z,r),[0:1]) \mid z\in\mbb{C}, r\geq0\}$. The point 
$((z,r),[0:1])\in S$ is mapped under $\mu$ to
\begin{equation*}
\frac{c}{1+\abs{z}^2+r^2}
\begin{pmatrix}
\frac{\abs{z}^2-r^2}{2}-\frac{1}{2}& r\ol{z}\\
rz & -\frac{\abs{z}^2-r^2}{2}+\frac{1}{2}
\end{pmatrix}.
\end{equation*}
Consequently, $\mu^{-1}(0)$ is non-empty if and $(c-1)\abs{z}^2=1$ for some 
$z\in\mbb{C}^*$, which is the case if and only if $c>1$. In summary, we see 
that, depending on $c>0$, the set of semistable points $X^{ss}_N$ can be empty 
or not.}
\end{ex}

\subsubsection{Proper moment maps}

Notice that Example~\ref{ex:varying_c} shows that for non-pro\-jec\-tive 
$G$-varieties in general the set of GIT-semistable points for the linearisation 
of the $G$-action in an ample line bundle $L$ and the set of semistable points 
with respect to a moment map $\mu$ computed using the curvature form of a 
Hermitian metric in the same line bundle $L$ do not have to coincide. In case 
the moment map under discussion is \emph{proper}, the two sets coincide by 
\cite[Theorem~2.18]{Sj2}. Hence, in the above example one could look for K\"ahler 
forms leading to proper moment maps that would then give a link to the algebraic 
theory and establish independence of semistabilty from the choice of the metric.

Looking at formula \eqref{eq:general_moment_map} one sees that a moment map of 
the most general form possible in the given situation is proper on 
$\mbb{C}^2\setminus\{0\}$ if and only if 
\begin{equation*}
\lim_{t\to0}\varphi'(t)t=\lim_{t\to\infty}\varphi'(t)t=\infty.
\end{equation*}
Since $t\mapsto\rho(e^t)=\varphi(e^{2t})$ is strictly increasing and strictly 
convex, we see that $t\mapsto 2\varphi'(e^{2t})e^{2t}$ is strictly increasing. 
Hence, $\lim_{t\to0}\varphi'(t)t=\infty$ is impossible. This proves that there 
is no proper ${\rm{SU}}(2)$-equivariant moment map on $\mbb{C}^2\setminus\{0\}$. 
Using compactness of $\mathbb{P}^1$, one can use this to conclude that none of 
the moment maps for the corresponding K\"ahler forms on $G/N \times 
\mathbb{P}^1$ is proper either, so that proper moment maps just do not exist in 
the situation at hand.

\subsubsection{Metrics arising from embedding into representations}

As we have seen in Section~\ref{sect:moment_maps_for_reps}, a natural choice in 
the situation at hand is to consider K\"ahler metrics on $G/N$ that are obtained 
by embedding this homogeneous space as an orbit in a $G$-representation space, 
and this is also the choice made in the  construction presented in 
Section~\ref{subsubsectio:def_of_ss}. We will see in 
Section~\ref{sect:properties} below that this leads to a number of desirable 
properties. However, also a restriction to this class of metrics does not lead 
to a common notion of semistability, as the following example shows. 

\begin{ex}{\rm
We continue the discussion at the end of Example~\ref{Ex:nonsurjectivemomentmap} 
and consider $G= {\rm{SL}}(3,\mbb{C})$ and $N=G_v$ the unipotent radical of a 
certain Borel subgroup of $G$. If we equip $G/N$ with the restriction of the 
flat K\"ahler metric for which the image of the moment map $\mu_V$ has 
complement with non-empty interior, and if we take $X=X_\alpha$ to be the 
$G$-flag manifold corresponding to the coadjoint orbit $\mathrm{Ad}^*(K)\acts 
\alpha$ through a point $\alpha \in \mathfrak{k}^*_{\rm reg}$ such that 
$-\alpha$ does not lie in the image of $\mu_V$, then the set of semistable 
points for the $G$-action on $G\times_NX$ is empty. On the other hand, as $N$ is 
a Grosshans subgroup of $G$, there exists a $G$-module $V'$
inducing a moment map $\mu_{V'}$ on $G/N$ whose image is a
$K$-invariant dense open subset of 
$\mathfrak{k}^*$, see Lemma~\ref{Lem:GrosshansMomentImage}. Without loss of 
generality, we may assume that the point $\alpha \in \mathfrak{k}^*_{\rm reg}$ 
chosen above fulfils $- \alpha \in \mu_{V'}(G/N)$. Then, for the K\"ahler 
metric on $G\times_NX$ induced by the second embedding the set of semistable 
points is not empty. }
\end{ex}

\subsubsection{Unipotent radicals of parabolic subgroups}

The next example shows that even in the case that we are able to embed $N$ as 
the unipotent radical of parabolic subgroups of two different semisimple groups 
$G_1$ and $G_2$ and hence, as explained in 
Remark~\ref{rem:forms_on_unipotent_radicals}, for each of the two embeddings 
there exists a very natural choice of a K\"ahler form on $G_j/N$, we cannot 
expect $X^{ss}_N$ to be independent of the group $G_j$.

\begin{ex} \label{ex:not_independent_of_the_group}{\rm
Let us consider the action of $N=\mbb{C}^2$ on $X=\mbb{P}_2$ given by the 
embedding
\begin{equation*}
N\hookrightarrow G_1={\rm{SL}}(3,\mbb{C}) = \mathrm{SU}(3)^\mathbb{C}= 
K_1^\mathbb{C},\quad(t,s)\mapsto
\begin{pmatrix}
1&0&t\\0&1&s\\0&0&1
\end{pmatrix}.
\end{equation*}
Taking the obvious Hamiltonian $G_1$-extension $Z_1=X$, we have $G_1\times_NX= 
G_1\times_NZ_1=(G_1/N)\times\mbb{P}_2$ with moment map $\mu=\mu_V+ 
\mu_{\mbb{P}_2}\colon(G_1/N)\times\mbb{P}_2\to\lie{k}_1^* = \mathfrak{su}(3)^*$. 
Since $N$ is embedded as the unipotent radical of a parabolic subgroup $P$ of 
$G_1$, we may consider the canonical affine completion $\ol{G_1/N}^\text{a}$ and 
equip it with the canonical K\"ahler form that is described in the paragraph 
before Remark~3.4 in~\cite{Kir}, 
cf.~Remark~\ref{rem:forms_on_unipotent_radicals}. 

The behaviour of the corresponding moment map $\mu_V$ on $\ol{G_1/N}^\text{a}$ 
is best understood in terms of its description as the universal 
$K^{(P)}$-imploded cross-section $(T^*K)^{K,K^{(P)}}_\text{impl}$, 
see~\cite[Definition~3.11]{Kir}. According to the discussion following 
Remark~3.13 in~\cite{Kir}, the $G_1$-orbits in $\ol{G_1/N}^\text{a}$ correspond 
to the strata \[(K_1\times\Ad^*(K_1^{(P)})\acts\sigma)/\negthickspace 
\approx_{K_1^{(P)}},\] where $\sigma$ runs through the open faces of 
$(\lie{t}_1^*)_+$. In particular, the open orbit $G_1/N$ is associated with the 
interior $\inn(\lie{t}_1^*)_+$ of $(\lie{t}_1^*)_+$. The description of the 
moment map $\mu_V$ given in~\cite[Theorem~3.12]{Kir} now implies that 
$\mu_V(G_1/N)$ is contained in $(\lie{k}_1)^*_{\reg}=\Ad^*(K_1)\acts 
\inn(\lie{t}_1^*)_+$. Since $\mu_{\mbb{P}_2}(\mbb{P}_2)$ does not intersect the 
interior of $(\lie{t}_1^*)_+$, the zero fibre of $\mu$ is empty, hence 
$X^{ss}_N=\emptyset$.

Now, let us consider the second embedding
\begin{equation*}
N\hookrightarrow G_2={\rm{SL}}(2,\mbb{C})\times{\rm{SL}}(2,\mbb{C}) =
(\mathrm{SU}(2) \times \mathrm{SU}(2))^\mathbb{C} \negthinspace= 
K_2^\mathbb{C},\quad(t,s)\mapsto \left(
\begin{pmatrix}
1&t\\0&1
\end{pmatrix},
\begin{pmatrix}
1&s\\0&1
\end{pmatrix}\right).
\end{equation*}
Here, $N$ is embedded as the unipotent radical of a Borel subgroup of $G_2$, and 
thus in particular again a Grosshans subgroup of $G_2$. As $G_2$-extension of 
the $N$-action on $X=\mbb{P}_2$ we choose the embedding
\begin{equation*}
\iota\colon\mbb{P}_2\hookrightarrow\mbb{P}_3,\quad
\iota([z_0:z_1:z_2])= [z_0:z_2:z_1:z_2],
\end{equation*}
which is $N$-equivariant for the $N$-action on $\mbb{P}_3$ given by
\begin{equation*}
N\hookrightarrow G_2\hookrightarrow{\rm{SL}}(4,\mbb{C}), \quad
(t,s)\mapsto
\begin{pmatrix}
1&t&0&0\\0&1&0&0\\0&0&1&s\\0&0&0&1
\end{pmatrix}.
\end{equation*}

A moment map $\mu_{\mbb{P}_3}\colon\mbb{P}_3\to\lie{su}(2)\oplus\lie{su}(2)$ 
for the $K_2$-action on $\mbb{P}_3$ with respect to the Fubini-Study metric is 
given by the explicit formula
$$
\mu_{\mbb{P}_3}([z_0:z_1:z_2:z_3])=
\frac{1}{\abs{z_0}^2+\dotsb+\abs{z_3}^2}
\left[
\begin{pmatrix}
\frac{\abs{z_0}^2-\abs{z_1}^2}{2}&\ol{z_0}z_1\\
z_0\ol{z_1}&-\frac{\abs{z_0}^2-\abs{z_1}^2}{2}
\end{pmatrix}\right.
\oplus\left.
\begin{pmatrix}
\frac{\abs{z_2}^2-\abs{z_3}^2}{2}&\ol{z_2}z_3\\
z_2\ol{z_3}&-\frac{\abs{z_2}^2-\abs{z_3}^2}{2}
\end{pmatrix}\right],
$$
see Example~\ref{ex:varying_c}. In order to determine the semistable locus 
$Y^{ss}_{G_2}(\mu_Y)$ in $Y=G_2\times_NX$ we consider the closed embedding
\begin{equation*}
Y=G_2\times_NX\hookrightarrow G_2\times_NZ\cong G_2/N\times\mbb{P}_3
\end{equation*}
and the moment map $\mu_Y=(\mu_V+\mu_{\mbb{P}_3})|_Y$. The canonical affine 
closure of \[G_2/N\cong (\mbb{C}^2\setminus\{0\})\times 
(\mbb{C}^2\setminus\{0\})\] is $V= \mathbb{C}^2 \oplus \mathbb{C}^2=\mbb{C}^4$, 
which we equip with the Hermitian structure 
$\frac{1}{2}\langle\cdot,\cdot\rangle_{st}$, where 
$\langle\cdot,\cdot\rangle_{st}$ is the standard Hermitian product of 
$\mbb{C}^4$. A direct calculation using the formulae given in 
Example~\ref{ex:varying_c} and the explicit expression for $\mu_{\mbb{P}_3}$ 
given above yields \[\mu_Y(eN, [0:0:1])= 
\mu_V((1,0),(1,0))+\mu_{\mbb{P}_3}([0:1:0:1])=0.\] Hence, we 
have $X^{ss}_N \not=\emptyset$.}
\end{ex} 

\subsection{Semistable points induced by affine completions of 
$G/N$}\label{subsect:completely_stable}

There is a further way to define $N$-semistable points
in $X$, less 
directly linked to the intrinsic geometry of $G/N$ and $X$. Instead 
of discussing the diagonal $G$-action on $G/N \times Z$ let us consider an 
affine completion $\ol{G/N}^{\text{a}}$ and consider the diagonal $G$-action 
on $\ol{G/N}^{\text{a}}\times Z$. Let $\ol{\iota}\colon X\hookrightarrow 
\ol{G/N}^{\text{a}}\times Z$ be the $N$-equivariant embedding and define 
$X^{\ol{ss}}_N [\omega_X]:=\ol{\iota}^{-1}((\ol{G/N}^{\text{a}}\times 
Z)^{ss}_G[\omega_V + \omega_Z])$. Then, $X^{\ol{ss}}_N[\omega_X]$ is an 
open $N$-invariant subset which contains but in general is strictly bigger than 
$X^{ss}_N[\omega_X]$. Analogously, we define $X^{\ol{s}}_N$ as 
$\ol{\iota}^{-1}((\ol{G/N}^{\text{a}}\times Z)^{s}_G[\omega_V + 
\omega_Z])$.

\begin{lem}
Let $(X,\omega_X)$ be a compact Hamiltonian $N$-manifold with a Hamiltonian 
$G$-extension $(Z,\omega_Z)$. If $N$ is a Grosshans subgroup of $G$, i.e., if 
$\mbb{C}[G]^N$ is finitely generated, then for the canonical affine completion
$\Spec\mbb{C}[G]^N$ of $G/N$ the set $X^{\ol{ss}}_N$ is non-empty.
\end{lem}

\begin{proof}
We already noticed in Section~\ref{sect:moment_maps_for_reps} that under the 
Grosshans assumption the corresponding moment map $\mu_V\colon 
\Spec\mbb{C}[G]^N\to\lie{k}^*$ is surjective. Every moment map 
$\mu=\mu_V+\mu_Z\colon \Spec\mbb{C}[G]^N\times Z\to \lie{k}^*$ thus has non-empty zero fibre.
\qed
\end{proof}

\subsection{Algebraic actions on projective manifolds}\label{subsect:algebraic_actions}

In this section, we study the following situation: let $X$ be a projective 
manifold and $N$ a unipotent group acting \emph{linearly} on $X$ in the sense 
that there exists a finite-dimensional $N$-representation $W$ such that the 
corresponding homomorphism $N \to \mathrm{GL}(W)$ embeds $N$ into a semisimple 
subgroup $G$ of $\mathrm{SL}(W)$, and an $N$-equivariant embedding $\iota\colon 
X\hookrightarrow \mathbb{P}(W)$. We will compare the moment map approach 
presented in earlier sections with the Geometric Invariant Theory approach of 
Doran--Kirwan \cite{DorKir}.

Consider the (very ample) line bundle $L_X:= 
\iota^*(\mathscr{O}_{\mathbb{P}(W)}(1) )$ on $X$, which is 
$N$-linearised by construction. Let $\langle \cdot, \cdot \rangle$ be a 
Hermitian inner product on $W$ and set $K := \mathrm{SU}(W, \langle \cdot, \cdot 
\rangle) \cap G$, so that $G = K^\mathbb{C}$. Endow $\mathbb{P}(W)$ and hence $X$ with 
the corresponding Fubini-Study K\"ahler form $\omega_{FS}$ and its restriction 
$\omega_X := \iota^*(\omega_{FS})$, respectively, so that $[\omega_X] = c_1(L_X) 
\in H^2 ( X, \,\mathbb{R})$. Note that $\mathbb{P}(W)$ is a 
Hamiltonian $G$-extension of $X$. Next, as suggested by the construction of 
semistable points with respect to $\omega_X$, we look at
 \[ \ol{\iota}\colon X \hookrightarrow  Y=G\times_NX\hookrightarrow G\times_N 
\mathbb{P}(W) \cong G/N\times \mathbb{P}(W) \hookrightarrow 
\ol{G/N}^{\mathrm{a}} \times \mathbb{P}(W), \]
cf.~Section~\ref{subsect:completely_stable}, and additionally at the 
$G$-linearised ample line bundle $ L:= \mathscr{O}_{\ol{G/N}^{\mathrm{a}}} 
\boxtimes \mathscr{O}_{\mathbb{P}(W)}(1)$. In this situation, we define
\[X^{\ol{s}}_N(L_X) := \ol{\iota}^{-1}((\ol{G/N}^{\mathrm{a}} \times 
\mathbb{P}(W))^{s}_G(L) )\]
to be the pre-image of the GIT-stable points for the $G$-action on 
$\ol{G/N}^{\mathrm{a}} \times \mathbb{P}(W)$ and the given 
linearisation, which is unique as 
$G$ is semisimple. The main comparison result regarding 
moment-map-semistability and GIT-semistability can now be formulated as 
follows: 

\begin{prop}\label{prop:comparison_mu_GIT}
 In the above situation, assume additionally that $\ol{G/N}^{\mathrm{a}}$ is 
normal. Then, we have
\begin{equation}\label{eq:comparison_mu_GIT}X^{ss}_N[\omega_X] = 
X^{\ol{s}}_N(L_X).
\end{equation}
\end{prop}

\begin{proof}
The inner product $\langle \cdot, \cdot \rangle $ induces a $K$-invariant 
Hermitian metric on $\mathscr{O}_{\mathbb{P}(W)}(1)$ such that $\frac{i}{2\pi}\, 
\times$ the curvature is $\omega_{FS}$. Using a $K$-invariant Hermitian metric 
on the trivial line bundle over $V \supset \ol{G/N}^{\mathrm{a}}$ with 
$\frac{i}{2\pi}\, \times$ curvature equal to $\omega_V$, we get a $K$-invariant 
Hermitian metric $h$ on $L \to \ol{G/N}^{\mathrm{a}} \times \mathbb{P}(W)$ with 
$\frac{i}{2\pi}\, \times$ curvature equal to $\omega_V + 
\omega_{FS}$. We are hence in the general situation of 
\cite[Section~2.2]{Sj2}\footnote{As $G$ is semisimple, the moment map computed 
there has to coincide with $\mu_{\ol{G/N}^{\mathrm{a}} \times \mathbb{P}(W)}$.}, 
with the exception that $\ol{G/N}^{\mathrm{a}} \times \mathbb{P}(W)$ is normal 
and not smooth, which does not affect Sjamaar's arguments\footnote{See also 
\cite{HausenHeinzner} for the result attributed to Roberts. Regarding \cite[Lemma 
2.16]{Sj2}, see also \cite{Extensionofsymplectic}.}. In particular, the compact 
complex space $(\ol{G/N}^{\mathrm{a}} \times 
\mathbb{P}(W))^{ss}_G(\mu_{\ol{G/N}^{\mathrm{a}} \times \mathbb{P}(W)}) \hq 
G$ is projective algebraic by Grauert's version of the Kodaira Embedding 
Theorem, see \cite[Theorem~2.17]{Sj2}. Moreover, we claim that
\begin{equation}\label{eq:semistable_the_same}(\ol{G/N}^{\mathrm{a}} \times 
\mathbb{P}(W))^{ss}_G(\mu_{\ol{G/N}^{\mathrm{a}} \times \mathbb{P}(W)}) = 
(\ol{G/N}^{\mathrm{a}} \times \mathbb{P}(W))^{ss}_G(L).
\end{equation}

In order to prove this, as the moment map $\mu_{\ol{G/N}^{\mathrm{a}} \times 
\mathbb{P}(W)}$ is proper we can follow the general line of argumentation 
presented in  \cite[proof of Theorem~2.18]{Sj2}: Since the possibly singular 
variety $\ol{G/N}^{\mathrm{a}} \times \mathbb{P}(W)$ is contained in $V\times 
\mathbb{P}(W)$, and since all differential geometric and symplectic objects are 
obtained by restriction, the computations regarding the relation between the 
norms of sections and (the norm square of) the moment map given in the first 
paragraph of \emph{loc.~cit.} continue to hold, so that for any $\mu$-semistable 
$p \in \ol{G/N}^{\mathrm{a}} \times \mathbb{P}(W)$ and any $G$-invariant section 
$s$ of $L$ over $(\ol{G/N}^{\mathrm{a}} \times 
\mathbb{P}(W))^{ss}_G(\mu_{\ol{G/N}^{\mathrm{a}} \times \mathbb{P}(W)})$, 
the restriction of the function $h(s , s)$ to the closure of $G \acts p$ inside 
$(\ol{G/N}^{\mathrm{a}} \times \mathbb{P}(W))^{ss}_G 
(\mu_{\ol{G/N}^{\mathrm{a}} \times \mathbb{P}(W)})$ takes on its maximum at 
the limit $F_\infty(p)$ under the gradient flow of 
$-\|\mu_{\ol{G/N}^{\mathrm{a}} \times \mathbb{P}(W)}\|^2$, from which we 
conclude that $s$ is bounded on $(\ol{G/N}^{\mathrm{a}} \times 
\mathbb{P}(W))^{ss}_G(\mu_{\ol{G/N}^{\mathrm{a}} \times \mathbb{P}(W)})$. 
Furthermore, an application of \cite[Proposition 7.6]{PaHq} shows that the set 
of $\mu$-semistable points is Zariski-open. Since in addition 
$\ol{G/N}^{\mathrm{a}} \times \mathbb{P}(W)$ is normal, it therefore follows 
from Riemann's Extension Theorem\footnote{In Sjamaar's setup the application of 
Riemann's Extension Theorem is not justified, since at this point the complement 
of the set of $\mu$-semistable points is not known to be small enough; e.g.,~it 
could contain interior points (in the Euclidean topology).} that $s$ extends to 
a $G$-invariant section over the whole of $\ol{G/N}^{\mathrm{a}} \times 
\mathbb{P}(W)$. The arguments for the two implications ``algebraically 
semistable implies analytically semi\-stable'' and ``analytically semi\-stable 
implies algebraically semistable'' can now be used without changes, proving 
\eqref{eq:semistable_the_same}. 
 
The analogous equality for stable points follows from the fact that on both 
sides, these are the ones for which the corresponding fibre of the quotient map 
consists of a single (closed) orbit. Intersecting with $\ol{\iota}(X)$ yields  
$X^{\ol{s}}_N[\omega_X] = X^{\ol{s}}_N(L_X)$, from which we conclude using 
Corollary~\ref{cor:completely_stable_gleich_semistable} proven in 
Section~\ref{subsect:geometric_quotients} below.
\qed
\end{proof}

\begin{rem}[Comparison of semistable points]{\rm
 In the given situation, Doran and Kirwan in \cite[Definition 5.1.6]{DorKir}  
define the set of \emph{GIT-semistable points} to be 
 \[X^{ss}_N(L_X) := X \cap Y^{ss}_G(\hat\iota^* \mathscr{O}_{G/N} \boxtimes 
\mathscr{O}_{\mathbb{P}(W)}(1)),\]
where $\hat\iota$ is given by $\eqref{eq:untwisting_the_action}$. In  general, 
this set will not coincide with $X^{ss}_N[\omega_X]$, as the following argument 
shows. Assume we had $X^{ss}_N[\omega_X] = X^{ss}_N(L_X)$. Since the latter set 
only depends on the $N$-action on $X$ and its lift to the $N$-linearised line 
bundle $L_X$, see \cite[Proposition 5.1.9]{DorKir}, the same would be true for 
$X^{ss}_N[\omega_X]$. In particular, $X^{ss}_N[\omega_X]$ would be independent 
of the chosen embedding $N \hookrightarrow G$ and of the chosen embedding $G/N 
\hookrightarrow V$ with (normal) affinisation $\ol{G/N}^{\mathrm{a}}$. This 
however would stand in contradiction to 
Example~\ref{ex:not_independent_of_the_group}.
 
In this direction, Equality~\eqref{eq:comparison_mu_GIT} gives the inclusion 
$X^{ss}_N[\omega_X] \subset X^{ss}_N(L_X) $, which in general is strict, as the 
gradient flow of the norm square of the moment map of a GIT-semistable point in 
$\hat{\iota}(Y)$ might converge to a point (in the zero fibre of the moment 
map) in the \emph{boundary} of $\hat{\iota}(Y)$ in $\ol{G/N}^{\mathrm{a}} 
\times \mathbb{P}(W)$.}
\end{rem}

\section{Properties of quotients by unipotent groups}\label{sect:properties}
 
We establish the existence of a compactifiable geometric quotient of the set of 
semistable points by the $N$-action that extends to a meromorphic map from $X$ 
to the compactification and carries a natural K\"ahler form obtained by 
symplectic reduction. We will use the notation established in 
Section~\ref{subsubsectio:def_of_ss}. 

\subsection{Existence of geometric quotients}\label{subsect:geometric_quotients}

As in the reductive case, sets of semistable points admit quotients, which in 
the unipotent case are automatically geometric, since unipotent groups cannot 
have properly semistable orbits by the following

\begin{lem}\label{Lem:AllOrbitsClosed}
Let $(X,\omega_X)$ be a Hamiltonian $N$-manifold. Then every $N$-orbit in 
$X^{ss}_N[\omega_X]$ is closed in $X^{ss}_N[\omega_X]$, i.e., we have 
$X^{ss}_N[\omega_X]=X^s_N[\omega_X]$.
\end{lem}

\begin{proof}
Consider the analytic Hilbert quotient $\pi_G\colon 
Y^{ss}_G[\omega_Y]\to Y^{ss}_G[\omega_Y]\hq G$. The fibre 
$\pi_G^{-1}(\pi_G^{}(x))$ is an affine $G$-variety, see 
\cite[Proposition 3.3.7]{HH2}. It hence follows from a classical result that 
every $N$-orbit is closed in $\pi_G^{-1}(\pi_G^{}(x))$ and hence in 
$Y^{ss}_G[\omega_X]$. The claim follows.
\qed
\end{proof}

\begin{cor}\label{cor:completely_stable_gleich_semistable}
In the situation of Section~\ref{subsect:completely_stable}, we have 
$X^{ss}_N[\omega_X] = X^{\ol{s}}_N[\omega_X]$.
\end{cor}

\begin{proof}
It follows from Lemma~\ref{Lem:AllOrbitsClosed} that every orbit in 
$Y^{ss}_G[\omega_Y]$ is closed in $Y^{ss}_G[\omega_Y]$. If $\Phi\colon Y 
\hookrightarrow \ol{G/N}^{\mathrm{a}} \times Z$ is the natural inclusion, we 
hence have
\[
 Y^{ss}[\omega_Y] = \{y \in Y \mid G\acts y \cap \mu_Y^{-1}(0) \neq \emptyset\} 
= \{y \in Y \mid G\acts \Phi(y) \cap \mu_{\ol{G/N}^{\mathrm{a}} \times 
Z}^{-1}(0) \neq \emptyset\}.
\]
As $\Phi$ restricted to $X \subset Y$ coincides with $\ol{\iota}$, the claim 
follows.
\qed
\end{proof}

\begin{thm}\label{thm:existence_of_geometric_quotient}
Let $(X, \omega_X)$ be a compact Hamiltonian $N$-manifold. Then, the set 
$X^{ss}_N[\omega_X]$ of semi\-stable points admits a geometric quotient 
$\pi\colon X^{ss}_N[\omega_X] \to X^{ss}_N[\omega_X] / N$ by the $N$-action. In 
fact, $\pi$ is a principal $N$-fibre bundle and $X^{ss}_N[\omega_X] / N =:Q$ is 
smooth.
\end{thm}

\begin{proof}
By the quotient theory for Hamiltonian actions of reductive groups, see 
Theorem~\ref{propertiesmomentumquotients}, the set of $G$-semistable points 
$Y^{ss}_G[\omega_Y] = G\acts X^{ss}_N[\omega_X]$ admits an analytic Hilbert 
quotient by the $G$-action. Moreover, by Lemma~\ref{Lem:AllOrbitsClosed}, every 
$G$-orbit in $Y^{ss}_G[\omega_Y]$ is closed there, hence the quotient 
$Y^{ss}_G[\omega_Y] \to Y^{ss}_G[\omega_Y]\hq G$ is in fact geometric. By 
construction of the twisted product, the restriction to $X^{ss}_N[\omega_X] 
\subset Y^{ss}_G[\omega_Y]$ yields the desired geometric quotient $\pi$.  

For 
every $x\in X^{ss}_N[\omega_X]$, the $G$-orbit is closed, hence the isotropy 
subgroup $G_x$ is reductive. On the other hand, as $x \in X$, the isotropy is 
contained in $N$, and hence unipotent. It follows that $G_x=N_x = \{e\}$, and 
hence that $\pi$ is a principal $N$-fibre bundle.
\qed
\end{proof}

\subsection{Compactifications of the quotient}

We will establish the existence of natural compactifications of the quotient 
$X^{ss}_N[\omega_X] / N$, \emph{which we assume to be non-empty in this section}.

Recall that the fundamental construction of Section~\ref{subsubsectio:def_of_ss} 
involves the choice of an embedding of ${G/N}$ into a Hermitian 
$K$-representation $V$ as a $G$-orbit $G\acts v$, see 
Equation~\eqref{eq:definition_of_omega_1}. This leads to an affine completion 
$\ol{G/N}^{\mathrm{a}} := \ol{G\acts v}$ of $G/N$ to which both the K\"ahler 
form and the moment map extend. Consider the composition $\Phi \colon Y 
\hookrightarrow \ol{G/N}^{\mathrm{a}} \times Z$ of the open embedding 
$G/N \times Z \hookrightarrow \ol{G/N}^{\mathrm{a}} \times Z$ with the embedding 
\eqref{eq:untwisting_the_action} used in the main construction.

\begin{prop}\label{prop:compactI}
The inclusion $Y = G\times_N X  \overset{\Phi}{\hookrightarrow} 
\overline{G/N}^{\mathrm{a}}\times Z$ induces an open embedding \[\phi\colon 
X^{ss}_N[\omega_X] / N \hookrightarrow \overline{Q}\] of $X^{ss}_N[\omega_X] / 
N= Q$ into a compact complex space $\overline{Q}$ such that $\overline{Q} 
\setminus \phi( Q )$ is analytic and nowhere dense.
\end{prop}

\begin{proof}
We claim that $\Phi(Y)$ is Zariski-open in its closure. For this, we first look 
at the compactification $V \times Z \hookrightarrow \mathbb{P}(V \oplus 
\mathbb{C}) \times Z$; by slight abuse of notation, the composition of $\Phi$ 
with this embedding will also be denoted by $\Phi$. Since the $G$-action on the 
compact K\"ahler manifold $Z$ is Hamiltonian, it is meromorphic, see 
Remark~\ref{rem:reductive_equivalence}. As in addition the $G$-action on $V$ is 
algebraic, the $G$-action on the compact K\"ahler manifold $\mathbb{P}(V \oplus 
\mathbb{C}) \times Z$ is meromorphic. Secondly, we notice that $\Phi(Y) = G 
\acts (\{eU\} \times \iota(X) )$, where $\iota\colon X\hookrightarrow 
Z$ is the extension map; i.e., $\Phi(Y)$ is the $G$-sweep of a compact complex 
submanifold of $\mathbb{P}(V \oplus \mathbb{C}) \times Z$. It therefore follows 
from \cite[Lemma 2.4(1)]{Fuj} that $\Phi(Y)$ is Zariski-open in its closure 
inside $\mathbb{P}(V \oplus \mathbb{C}) \times Z$, and hence it is Zariski-open 
in its closure inside $\ol{G/N}^{\mathrm{a}} \times Z$. We denote this closure 
by $\ol{Y}$.

By Theorem~\ref{Thm:MomentImage} the moment map $\mu_V$ is proper on 
$\ol{G/N}^{\mathrm{a}}$. It follows that the moment map for the action of $K$ on 
$\ol{G/N}^{\mathrm{a}} \times Z$ and hence the restriction $\mu_{\ol{Y}}\colon 
\ol{Y} \to \mathfrak{k}^*$ of this moment map to the analytic subset $\ol{Y} 
\subset \ol{G/N}^{\mathrm{a}} \times Z$ is likewise proper. Recalling the 
construction of the K\"ahler form $\omega_Y$ and the associated moment map 
$\mu_Y$, we can summarise the situation in the following commutative diagram:
\[\begin{xymatrix}{
   Y   \ar@{^(->}[r] \ar[rd]_{\mu_{Y}}&  \ol{Y} 
\ar[d]^{\mu_{\ol{Y}}}\ar@{^(->}[r] & \ol{G/N}^{\mathrm{a}} \times Z 
\ar@{^(->}[r] \ar[d] & V \times Z  \ar[ld]^{\mu_V + \mu_Z}\\
		 &         \mathfrak{k}^*          \ar[r]^=      &		
		 \mathfrak{k}^*	. &
}
  \end{xymatrix}\]
Here, in the first line, the first inclusion is open and the other two 
inclusions are closed. 

Since $\mu_{\ol{Y}}$ is proper, its zero fibre is compact, and hence the 
associated analytic Hilbert quotient $\ol{Y}^{ss}_G(\mu_{\ol{Y}}) \hq G \simeq 
\mu^{-1}_{\ol{Y}}(0) / K$ is a compact complex space, which in fact comes with a 
natural closed embedding into the (non-compact) analytic Hilbert quotient $(V 
\times Z)^{ss}_G(\mu_V + \mu_Z)\hq G$. As the inclusion $Y \hookrightarrow 
\ol{Y}$ has Zariski-open image, and as every $G$-orbit in $Y^{ss}_G(\mu_Y) = 
Y^{ss}_G[\omega_Y]$ is closed by the argument in the proof of Theorem 
\ref{thm:existence_of_geometric_quotient}, the inclusion $Y^{ss}_G[\omega_Y] 
\hookrightarrow \ol{Y}^{ss}_G(\mu_{\ol{Y}})$ is Zariski-open and saturated with 
respect to the quotient map $\ol{\pi}\colon \ol{Y}^{ss}_G(\mu_{\ol{Y}}) \to 
\ol{Y}^{ss}_G(\mu_{\ol{Y}}) \hq G$. It therefore induces the desired 
Zariski-open embedding \[\phi\colon Q=X^{ss}_N[\omega_X]/N \cong 
Y^{ss}_G[\omega_Y]/G \hookrightarrow \ol{Y}^{ss}_G(\mu_{\ol{Y}})\hq G =: \ol{Q}  
\]
into the compact complex space $\ol{Q}$.
\qed
\end{proof}

\subsection{Zariski-openness of semistable points and meromorphic extension of 
the quotient map}

While there is no general result for analyticity of the complement of the set 
of semi\-stable points in a Hamiltonian $G$-manifold with non-proper moment 
map, in our setup this can be shown by hand.

\begin{thm}\label{thm:Zopen}
Let $(X, \omega_X)$ be a compact Hamiltonian $N$-manifold. Then, the set 
$X^{ss}_N[\omega_X]$ of semi\-stable points is Zariski-open in $X$. Moreover, 
the quotient map $\pi\colon X^{ss}_N[\omega_X] \to X^{ss}_N[\omega_X] / N$ 
extends to a meromorphic map\footnote{The reader is referred to \cite[Chapter 6, 
Section 3]{Whitney} for an in depth discussion of meromorphic mappings between 
complex spaces.} $\pi\colon X \dasharrow \overline{Q}$ to the compact complex 
space $\overline{Q}$ constructed in the proof of 
Proposition~\ref{prop:compactI}. 
\end{thm} 

\begin{proof}
By part (1) of Theorem~\ref{propertiesmomentumquotients}, $X^{ss}_N[\omega_X]$ 
is open in the Euclidean topology.  Let $\pi_F \colon X \dasharrow Q_F$ be a 
Fujiki quotient of $X$ by the $N$-action, whose existence is guaranteed by 
Proposition~\ref{Chow}, and let $\Gamma \subset X \times Q_F$ be the graph. In 
particular, $\Gamma$ is an $N$-invariant, irreducible, compact analytic subset 
of $X \times Q_F$, where $N$ acts only on the first factor. Embedding $X \times 
Q_F$ into $Y\times Q_F$ and further into $\mathbb{P}(V \oplus \mathbb{C}) \times 
Z \times Q_F$ as in the proof of Proposition~\ref{prop:compactI} we can 
interpret $\Gamma$ as an $N$-invariant subvariety in $\mathbb{P}(V \oplus 
\mathbb{C}) \times Z \times Q_F$. Using again that the $G$-action on the latter 
space is meromorphic, we conclude that $\hat{\Gamma}:= \ol{G \acts \Gamma}$ 
is Zariski-open in its closure in $\ol{Y} \times Q_F$, and in particular 
irreducible. On $Y \subset \ol{Y}$ it is the graph of the $G$-invariant 
extension of the $N$-invariant meromorphic map $\pi_F$ from $X$ to $Y = G 
\times_N X$. It follows that $\hat{\Gamma}$ is the graph of a $G$-invariant 
meromorphic map from $\ol{Y}$ to $Q_F$, which we will call $\hat{\pi}_F$. 
The graph of the restriction of $\hat{\pi}_F$ to 
$\ol{Y}^{ss}_G(\mu_{\ol{Y}})$ is equal to $\hat{\Gamma}^{\circ} := 
\hat{\Gamma} \cap (\ol{Y}^{ss}_G(\mu_{\ol{Y}}) \times Q_F)$. Now, 
$\ol{Y}^{ss}_G(\mu_{\ol{Y}}) \times Q_F$ admits an analytic Hilbert quotient by 
the $G$-action, namely $\Pi = \ol{\pi} \times \id_{Q_F}\colon 
\ol{Y}^{ss}_G(\mu_{\ol{Y}}) \times Q_F  \to \overline{Q} \times Q_F$. As 
$\hat{\Gamma}^\circ$ is a $G$-invariant, irreducible analytic subset of 
$\ol{Y}^{ss}_G(\mu_{\ol{Y}}) \times Q_F $, its image 
$\hat{\Gamma}^\circ_{\rm red} = \Pi(\hat{\Gamma}^\circ)$ is an 
irreducible analytic subset of $\ol{Q} \times Q_F$ by the basic properties of 
analytic Hilbert quotients listed in Section~\ref{subsect:aHq_properties}. 
 
On the one hand, as orbits through points in $Y^{ss}_G(\mu_Y) \subset  
\ol{Y}^{ss}_G(\mu_{\ol{Y}})$ are closed, $\hat{\Gamma}^\circ_{\rm red}$ 
defines a meromorphic map $\pi_{F, \rm red}$ from $\ol{Q}$ to $Q_F$, cf.~the 
argument given in the proof of \cite[Proposition~4.5]{PaHq}. 
On the other hand, 
consider the open subset $U := U_F \cap X^{ss}_N[\omega_X]$, 
cf.~Proposition~\ref{Chow}. As this set is $N$-invariant, and since both 
$\pi_F|_U$ and $\pi|_U$ are geometric quotients for the $N$-action on $U$ by 
Proposition~\ref{Chow} and Theorem~\ref{thm:existence_of_geometric_quotient}, 
respectively, the respective images $\pi_F(U) \subset Q_F$ and $\pi(U) \subset Q 
\subset \ol{Q}$ are biholomorphic via $\pi_{F, \rm red}$. It follows that 
$\pi_{F, \rm red}\colon \ol{Q} \dasharrow Q_F$ is bimeromorphic. From this, we 
conclude that $(\pi_{F, \rm red})^{-1} \circ \pi_F \colon X \dasharrow \ol{Q}$ 
is a meromorphic extension of $\pi$ and that there are Zariski-open, dense 
subsets $\overline{\Omega} \subset \ol{Q}$ and $\Omega_F \subset Q_F$ that are 
biholomorphic via $\pi_{F, \rm red}$. By shrinking $U_F$ if necessary, we may 
assume that $\Omega_F = \pi_F(U_F)$ and $\ol{\Omega} \subset Q \subset \ol{Q}$.
The situation is hence summarised by the following commutative diagram
 \[\begin{xymatrix}
    { U_F \ar@{^(->}[r]\ar@{->>}[d]_{\pi_F} & X^{ss}_N[\omega_X] \ar@{->>}[d]^{\pi}&   \\
      \Omega_F \ar@{^(->}[r]^>>>>>>{(\pi_{F, \rm red})^{-1}}&  Q  \ar@{^(->}[r] & \ol{Q}.   }
   \end{xymatrix}
\]
In particular, the Zariski-open subset $U_F$ is contained in 
$X^{ss}_N[\omega_X]$. Since $X$ is compact, using a Noetherian induction 
argument applied to the analytic subset $X' :=X \setminus U_F$ we conclude that 
$X^{ss}_N[\omega_X]$ and hence $X\setminus X^{ss}_N[\omega_X]$ is constructible 
in the Zariski-topology of $X$. As we know from the start that $X\setminus 
X^{ss}_N[\omega_X]$ is closed in the Euclidean topology of $X$, the claim 
follows.
\qed
\end{proof}

\subsection{Reduced K\"ahler structure on the quotient}

We will show that using a symplectic reduction procedure the quotient 
$X^{ss}_N[\omega_X] / N$ can be endowed with a K\"ahler form naturally induced 
from $\omega_X$. This form extends to the compactification $\ol{Q}$ and its 
class pulls back under $\pi$ to the class of $[\omega_X]$ on 
$X^{ss}_N[\omega_X]$.

\begin{thm}\label{thm:reducedKaehler}
In the setup of Proposition~\ref{prop:compactI}, there exists a K\"ahler 
structure\footnote{See \cite[Sections 3.1 and 3.2]{PaHq} for the basic 
definitions regarding K\"ahler structures on (singular) complex spaces.} 
$\omega_{\ol{Q}}$ on the compact complex space $\ol{Q}$ whose restriction 
$\omega_Q = \omega_{\ol{Q}}|_Q^{}$ to $Q \hookrightarrow \ol{Q}$ is smooth and 
fulfils
 \[[\pi^*{\omega_Q}] = [\omega_X|_{X^{ss}_N[\omega_X]}] \in 
H^2(X^{ss}_N[\omega_X],\, \mathbb{R}).\]
\end{thm}

\begin{proof}
Once again, recall our setup in the following commutative diagram
 \begin{equation}\label{eq:big_diagram}
 \begin{gathered}
\begin{xymatrix}{
  X^{ss}_N[\omega_X] \ar@{^(->}[r] \ar@{->>}[d]_\pi &  Y^{ss}_G[\omega_Y]
\ar@{->>}[d]  \ar@(ur,ul)[rrr]^\psi \ar@{^(->}[r] & 
\ol{Y}^{ss}_G[\hat{\omega}_{Z}] \ar@{^(->}[r] \ar@{->>}[d]^{\ol{\pi}} &
\ol{G/N}^{\mathrm{a}} \times Z \ar@{^(->}[r]& V \times Z \ar[d]^{\mathrm{pr}_Z}
\\
  Q \ar[r]^<<<<<<<\cong& Y^{ss}_G[\omega_Y]\hq G \ar@{^(->}[r]  & \ol{Q} & & Z.
}
  \end{xymatrix}
 \end{gathered}
 \end{equation}
By applying the K\"ahlerian reduction procedure of \cite{Extensionofsymplectic} 
to $\ol{Y}^{ss}_G[\hat{\omega}_{Z}]$ and to the quotient $\ol{Q}$, we obtain 
a K\"ahlerian structure $\omega_{\ol{Q}}$ on $\ol{Q}$ induced by restricting 
local $K$-invariant potentials of $\hat{\omega}_Z$ to $\mu_{\ol{Y}}^{-1}(0)$ 
and by the homeomorphism $\mu_{\ol{Y}}^{-1}(0)/K \simeq \ol{Q}$, 
cf.~Theorem~\ref{propertiesmomentumquotients}. We denote the restriction of 
$\omega_{\ol{Q}}$ to $Q$ by $\omega_Q$. 

In order to show that $\omega_Q$ is smooth, we first note that 
$Y^{ss}_G[\omega_Y]  \subset \ol{Y}^{ss}_G[\hat{\omega}_{Z}]$ is smooth and 
$\ol{\pi}$-saturated, and secondly recall the observation made in the proof of 
Theorem~\ref{thm:existence_of_geometric_quotient} above that the $G$-action on 
$Y^{ss}_G[\omega_Y]$ is free. Therefore, it follows from the construction of the 
reduced K\"ahler form $\omega_{\ol{Q}}$, see \cite[Lemma 2 on page 132 and the 
proof on pages 133/134]{Extensionofsymplectic} and also compare with 
\cite[Theorem 2.10]{Sj2}, that in the fundamental commutative diagram 
\begin{equation}\label{eq:fundamental_diagram}
\begin{gathered}
 \begin{xymatrix}{\mu_Y^{-1}(0)\ar@{->>}[d]_{\pi_K} \ar@{^(->}[r]^{\tau}&Y^{ss}_G[\omega_Y] \ar@{->>}[d]^{\ol{\pi}}\\
  \mu_Y^{-1}(0)/K \ar[r]_<<<<<{\simeq}^<<<<<{\tau_{\mathrm{red}}} & Q} 
\end{xymatrix}
\end{gathered}
\end{equation}
the fibre $\mu_Y^{-1}(0)$ is smooth, the $K$-action on $\mu_Y^{-1}(0)$ is free, 
and that the K\"ahler structure $\omega_Q$ is smooth and fulfils the 
``symplectic reduction'' equation
\begin{equation}\label{eq:symp_red}
 \tau^*(\ol{\pi}^*\omega_Q) = \pi_K^*(\tau_{\mathrm{red}}^*\omega_Q) = 
\omega_Y|_{\mu_Y^{-1}(0)} = \tau^*(\omega_Y|_{Y^{ss}_G[\omega_Y]}).
\end{equation}

More is true. Since $Y^{ss}_G[\omega_Y]$ is $\ol{\pi}$-saturated, and since the 
moment map $\mu_{\ol{Y}}\colon \ol{Y} \to \mathfrak{k}^*$ is proper as observed 
in the proof of Proposition~\ref{prop:compactI}, the moment map $\mu_Y$ is 
\emph{admissible} in the sense that the gradient flow $F_t$ of $- \|\mu_Y \|^2$ 
through any point $p \in Y^{ss}_G[\omega_Y]$ exists for all times, 
cf.~\cite[\S9]{Kir2}, and hence there exists a continuous retraction of 
$Y^{ss}_G[\omega_Y]$ to $\mu_Y^{-1}(0)$ defined by $z \mapsto \lim_{t \to 
\infty} F_t(z)$, see \cite[page 109]{Sj2} and the references given there, as 
well as \cite{Ler}. In particular, the inclusion displayed in the first line of 
Diagram~\eqref{eq:fundamental_diagram} induces an isomorphism between de Rham 
cohomology groups,
\begin{equation*}
 \tau^*\colon H^2(Y^{ss}_G[\omega_Y],\, \mathbb{R} ) 
\overset{\cong}{\longrightarrow} H^2(\mu_Y^{-1}(0), \mathbb{R} ).
\end{equation*}
Equation~\eqref{eq:symp_red} therefore implies that 
\begin{equation}\label{eq:1}
 [\ol{\pi}^*\omega_Q] = [\omega_Y|_{Y^{ss}_G[\omega_Y]}] \in 
H^2(Y^{ss}_G[\omega_Y], \, \mathbb{R} ).
\end{equation}

In addition, from the right hand part of Diagram~\eqref{eq:big_diagram}, from 
Equations~\eqref{eq:definition_of_omega_1} and \eqref{eq:definition_of_omega_2}, 
and from the the fact that the de Rham cohomology class of $\omega_V$ is trivial 
we infer that
\begin{equation*}
[\omega_Y|_{Y^{ss}_G[\omega_Y]}] = [\psi^* (\mathrm{pr}^*_Z 
(\omega_Z))] \in H^2(Y^{ss}_G[\omega_Y],\, \mathbb{R}),
\end{equation*}
so that \eqref{eq:1} becomes 
\begin{equation*}
[\ol{\pi}^*\omega_Q] =[\psi^* (\mathrm{pr}^*_Z (\omega_Z))] 
\in H^2(Y^{ss}_G[\omega_Y],\, \mathbb{R}).
\end{equation*}
Finally, using this, the left hand part of Diagram~\eqref{eq:big_diagram}, and 
the fact that $Z$ as a $G$-extension of $X$ fulfils \eqref{eq:pullback} we 
conclude 
\[[\pi^*{\omega_Q}] = [\omega_X|_{X^{ss}_N[\omega_X]}] \in 
H^2(X^{ss}_N[\omega_X],\, \mathbb{R}).\]

\vspace{-0.65cm}\qed
\end{proof}

\providecommand{\bysame}{\leavevmode\hbox to3em{\hrulefill}\thinspace}
%
%

\bibliographystyle{amsalpha}
\bibliographymark{References}
\def\cprime{$'$}

\end{document}